\documentclass[11pt, twoside]{article}
\pdfoutput=1

\usepackage[margin=1in]{geometry}

\usepackage{amssymb}
\usepackage[mathscr]{eucal}
\usepackage{bm}
\usepackage{cite}

\newcommand{\norm}[1]{\left\|#1\right\|}
\newcommand{\innOm}[2]{(#1,#2)_\Omega}
\newcommand{\innGa}[2]{\langle #1,#2\rangle_\Gamma}

\usepackage{graphicx}
\usepackage{amsmath,amsfonts,amssymb}
\usepackage[mathscr]{eucal}
\usepackage{bm}

\usepackage{amsthm}
\usepackage{cite}
\usepackage{hyperref}

\usepackage[nameinlink]{cleveref}
\numberwithin{equation}{section}

\newtheorem{theorem}{Theorem}[section]
\newtheorem{lemma}[theorem]{Lemma}
\newtheorem{proposition}[theorem]{Proposition}

\newtheorem{corollary}[theorem]{Corollary}

\theoremstyle{definition}

\theoremstyle{remark}
\newtheorem{remark}[theorem]{Remark}

\usepackage{fancyhdr}
\begin{document}
	
	\title{A New HDG Method for Dirichlet Boundary Control of Convection Diffusion PDEs II: Low Regularity}

	\author{Weiwei Hu%
		\thanks{Department of Mathematics, Oklahoma State
			University, Stillwater, OK (weiwei.hu@okstate.edu). W.~Hu was supported in part by a postdoctoral fellowship for the annual program on Control Theory and its Applications at the Institute for Mathematics and its Applications (IMA) at the University of Minnesota.}%
		\and
		Mariano Mateos%
		\thanks{Dpto. de Matem\'aticas. Universidad de Oviedo, Campus de Gij\'on, Spain (mmateos@uniovi.es).  M.\ Mateos was supported by the Spanish Ministerio de Econom\'{\i}a y Competitividad under project MTM2014-57531-P.}%
		\and
		John~R.~Singler%
		\thanks{Department of Mathematics
			and Statistics, Missouri University of Science and Technology,
			Rolla, MO (\mbox{singlerj@mst.edu}, ywzfg4@mst.edu). J.~Singler and Y.~Zhang were supported in part by National Science Foundation grant DMS-1217122.  J.~Singler and Y.~Zhang thank the IMA for funding research visits, during which some of this work was completed.}
		\and
		Xiao Zhang%
		\thanks{College of Mathematics, Sichuan University,
		Chengdu, China (zhangxiaofem@163.com). X.~Zhang thanks Missouri University of Science and Technology for hosting him as a visiting scholar; some of this work was completed during his research visit.}
		\and
		Yangwen Zhang%
		\footnotemark[3]
	}
	
	\date{}
	
	\maketitle

	\maketitle
	
	%% ------------------------------------------------------------------
	%% ABSTRACT
	%% ------------------------------------------------------------------
		\begin{abstract}
			In the first part of this work, we analyzed a Dirichlet boundary control problem for an elliptic convection diffusion PDE and proposed a new hybridizable discontinuous Galerkin (HDG) method to approximate the solution.  For the case of a 2D polygonal domain, we also proved an optimal superlinear convergence rate for the control under certain assumptions on the domain and on the target state.  In this work, we revisit the convergence analysis without these assumptions; in this case, the solution can have low regularity and we use a different analysis approach.  We again prove an optimal convergence rate for the control, and present numerical results to illustrate the convergence theory.
		\end{abstract}

\section{Introduction}
In Part I of this work \cite{HuMateosSinglerZhangZhang1}, we considered the following  Dirichlet boundary control problem:  Minimize the cost functional
\begin{align}
J(u)=\frac{1}{2}\| y- y_{d}\|^2_{L^{2}(\Omega)}+\frac{\gamma}{2}\|u\|^2_{L^{2}(\Gamma)}, \quad \gamma>0, \label{cost1}
\end{align}
subject to the elliptic convection diffusion equation
\begin{equation}\label{Ori_problem}
\begin{split}
-\Delta y+\bm \beta\cdot\nabla y&=f \quad\text{in}~\Omega,\\
y&=u\quad\text{on}~\partial\Omega,
\end{split}
\end{equation}
where $ f \in L^2(\Omega) $, the vector field $\bm{\beta}$ satisfies
\begin{align}\label{beta_con}
\nabla\cdot\bm{\beta} \le 0,
\end{align}
and $\Omega\subset \mathbb{R}^{d} $ $ (d\geq 2)$ is a Lipschitz polyhedral domain with boundary $\Gamma = \partial \Omega$.

Many researchers have considered the numerical approximation of optimal control problem for convection diffusion equations \cite{MR2302057,MR2595051,MR2486088,MR2178571,MR2068903,MR2550371} and also optimal Dirichlet boundary control problems for the Poisson equation and other PDEs \cite{MR2272157,MR2558321,MR3070527,MR2347691,MR3317816,MR3614013,MR2020866,MR2567245,MR1632548,ravindran2017finite,MR3641789,MR2806572,ApelMateosPfeffererRosch17,MR1135991,MR1145711,MR1874072}.  However, the authors are unaware of any theoretical and numerical works in the literature concerning the above problem.  Progress on this problem is an important step towards the analysis and approximation of Dirichlet boundary control problems for fluid flows.

Formally, the optimal control $ u \in L^2(\Gamma) $ and the optimal state $ y \in L^2(\Omega) $ minimizing the cost functional satisfy the optimality system
\begin{subequations}\label{eq_adeq}
	\begin{align}
	-\Delta y+\bm \beta\cdot\nabla y &=f\qquad\quad\text{in}~\Omega,\label{eq_adeq_a}\\
	y &= u\qquad\quad\text{on}~\partial\Omega,\label{eq_adeq_b}\\
	-\Delta z-\nabla\cdot(\bm{\beta} z) &=y-y_d\quad  \text{in}~\Omega,\label{eq_adeq_c}\\
	z&=0\qquad~ \ \  \ \text{on}~\partial\Omega,\label{eq_adeq_d}\\
	\nabla z \cdot\bm n - \gamma  u&=0\qquad\quad \  \text{on}~\partial\Omega.\label{eq_adeq_e}
	\end{align}
\end{subequations}
In Part I, we showed in the 2D case that the optimal control is indeed determined by a weaker formulation of the above optimality system and we proved a regularity result for the solution.

We also introduced a new hybridzable discontinuous Galerkin (HDG) method to approximate the solution of the optimality system, and obtained an optimal superlinear convergence rate for the control.  However, there are two main restrictions for our convergence results in Part I.  First, we assumed the largest interior angle $ \omega $ of the convex polygonal domain belongs to $[\pi/3,2\pi/3)$.  Second, we assumed the desired state $y_d$ is in $ H^{s}(\Omega)$ for some $s>1/2$.  When one of these conditions is not satisfied, the optimal control can have low regularity, i.e., $ u \in H^{r_u}(\Gamma) $ for some $ r_u < 1 $.  We briefly review the regularity theory and the new HDG algorithm in \Cref{sec:Regularity and HDG Formulation}.

In this work, we use techniques from \cite{MR3508837,LiXie16_SINUM} to remove the restrictions on the largest interior angle $ \omega $ of the convex domain $\Omega$ and the desired state $y_d$.  Specifically, in \Cref{sec:analysis} we obtain optimal convergence rates for the control when $\omega\in [\pi/3,\pi)$ and $y_d\in H^{s}(\Omega)$ for some $s \geq 0$.  We illustrate the low regularity convergence theory with numerical results in \Cref{sec:numerics}.

\section{Background: Regularity and HDG Formulation}
\label{sec:Regularity and HDG Formulation}

To begin, we briefly review the regularity results for the optimal control problem and the new HDG method from Part I.

\subsection{Optimal Control Problem: Regularity}
\label{sec:regularity}

As in Part I, we use the standard notation $W^{m,p}(\Omega)$ for Sobolev spaces on $ \Omega $, and let $\|\cdot\|_{m,p,\Omega}$ and $|\cdot|_{m,p,\Omega}$ denote the Sobolev norm and seminorm.  We let $H^{m}(\Omega)$ denote the Sobolev space when $ p = 2 $ with norm $\|\cdot\|_{m,\Omega}$ and seminorm $|\cdot|_{m,\Omega}$.  Also, set $H_0^1(\Omega)=\{v\in H^1(\Omega):v=0 \;\mbox{on}\; \partial \Omega\}$ and $ H(\text{div},\Omega) = \{\bm{v}\in [L^2(\Omega)]^2, \nabla\cdot \bm{v}\in L^2(\Omega)\} $.  We denote the $L^2$-inner products on $L^2(\Omega)$ and $L^2(\Gamma)$ by
\begin{align*}
(v,w)_{\Omega} &= \int_{\Omega} vw  \quad \forall v,w\in L^2(\Omega),\\
\left\langle v,w\right\rangle_{\Gamma} &= \int_{\Gamma} vw  \quad\forall v,w\in L^2(\Gamma).
\end{align*}

For the analysis of the optimal control problem, we considered the following scenario in Part~I.  Suppose $\Omega$ is a convex polygonal domain, and let $\omega$ denote its largest interior angle.  We have $1<\pi/\omega\leq 3$.  We assume $ \bm \beta $ satisfies
\begin{equation}\label{eqn:beta_assumptions1}
\bm \beta \in [L^\infty(\Omega)]^d,  \quad  \nabla\cdot\bm{\beta}\in L^\infty(\Omega),  \quad  \nabla \cdot \bm \beta \leq 0,  \quad  \nabla \nabla \cdot \bm{\beta}\in[L^2(\Omega)]^d.
\end{equation}
The mixed weak form of the formal optimality system \eqref{eq_adeq_a}-\eqref{eq_adeq_e} is
\begin{subequations}
	\begin{align}
	\innOm{ {\bm q}}{\bm r} -\innOm{ y}{\nabla\cdot \bm r} + \innGa{ u}{\bm r\cdot \bm n} &= 0,\label{mixed_a}\\%\label{ME13a}\\
	\innOm{\nabla\cdot( {\bm q}+\bm{\beta}  y)}{w} -\innOm{ y\nabla\cdot\bm{\beta}}{w}  &= \innOm{f}{w},\label{mixed_b}\\%\label{ME13b}\\
	\innOm{ {\bm p}}{\bm r}-\innOm{ z}{\nabla\cdot \bm r} &= 0,\label{mixed_c}\\% \label{ME13c}\\
	\innOm{\nabla\cdot( {\bm p}-\bm{\beta} z)}{w} &= \innOm{ y- {y_d}}{w},\label{mixed_d}\\%\label{ME13d}\\
	\innGa{\gamma  u+ {\bm p}\cdot\bm n}{\mu} &=  0,\label{mixed_e}%\label{ME13e}
	\end{align}
	for all $(\bm r,w,\mu)\in H(\textup{div},\Omega)\times L^2(\Omega)\times L^2(\Gamma)$.
\end{subequations}
Also, we assume $ f = 0 $ for the regularity theorem below; nonzero forcing can be treated by a simple change of variables as in \cite[pg.\ 3623]{AMPR-2015}.

We proved the following regularity theorem in Part I \cite{HuMateosSinglerZhangZhang1}.
\begin{theorem}\label{MT210}
	Assume $\Omega$ is convex and $ f = 0 $.  If $y_d\in H^{t^*}(\Omega)$ for some $0\leq t^*<1$, then the optimal control problem has a unique solution $  u\in L^2(\Gamma)$ and $ u $ is uniquely determined by the optimality system \eqref{mixed_a}-\eqref{mixed_e}.  Moreover, for any $ s \geq 1/2 $ satisfying $s\leq \frac 12 +t^*$ and $s<\min\{\frac 3 2,\frac{\pi}{\omega}-\frac 1 2\}$, we have $  u\in H^s (\Gamma)$,
	\begin{align*}
	({\bm p}, y,  z) &\in [H^{s+\frac 12}(\Omega)]^d\times H^{s+\frac 12}(\Omega)\times (H^{s+\frac 32}(\Omega)\cap H_0^1(\Omega)),\\
	{\bm q} &\in  [H^{s-\frac 12}(\Omega)]^d \cap H(\mathrm{div},\Omega).
	\end{align*}
\end{theorem}

\Cref{MT210} implies the regularity of the solution of the optimality system \eqref{mixed_a}-\eqref{mixed_e} depends on the desired state $y_d$ and the domain $\Omega$.  As is known, solutions to Dirichlet boundary control problems can have low regularity; this causes difficulty for numerical analysis.

In Part I \cite{HuMateosSinglerZhangZhang1}, for the numerical analysis of the new HDG method we assumed $ \Omega $ is convex, $ y_d \in H^{t^*}(\Omega) $ for some $ t^* \in (1/2,1) $, and $ \pi/3 <\omega < 2\pi/3 $.  These assumptions give high regularity for the optimal control, i.e., $ u \in H^{r_u}(\Gamma) $ for some $ r_u \in (1, 3/2) $.  Furthermore, the assumptions give $\bm{q} \in H^{r_{\bm q}}(\Omega)$ with $ r_{\bm q} > 1/2 $, which guarantees $ \bm q $ has a well-defined trace on the boundary $ \Gamma $.  We used this property in the HDG convergence analysis.

In this paper we again assume $ \Omega $ is convex, but we remove the restrictions on the desired state and the largest interior angle for the numerical analysis; i.e., we only require $ t^* \in [0,1) $ and $\pi/3 \leq \omega<\pi $.  In this case, the regularity of the optimal control can be low, i.e., $ u \in H^{r_u}(\Gamma) $ for some $ r_u \in [1/2, 1) $, and $ \bm q $ is no longer guaranteed to have a well-defined $ L^2 $ boundary trace; however, the optimality system \eqref{mixed_a}-\eqref{mixed_e} can be understood in a standard weak sense.

\subsection{The HDG Formulation}
%\subsection{Mesh and Approximation spaces}
\label{sec:HDG}

For the HDG method, we assume $ \Omega $ is a polyhedral domain with $ d \geq 2 $ that is not necessarily convex.  We use the same notation from Part I \cite{HuMateosSinglerZhangZhang1} to describe the method.  For more information about HDG methods, see, e.g., \cite{MR2485455,MR2772094,MR2513831,MR2558780,MR2796169, MR3626531,MR3522968,MR3463051,MR3452794,MR3343926,ChenQiuShi17,MR3440284,QiuShi16NS}.

Let $\mathcal{T}_h$ be a collection of disjoint elements that partition $\Omega$, and let $\partial \mathcal{T}_h$ be the set $\{\partial K: K\in \mathcal{T}_h\}$.  For the analysis, we assume $\mathcal{T}_h$ is a conforming triangulation of $\Omega$.  Denote the elements of $\mathcal{T}_h$ by $K$ and the faces of $K$ by $e$. Denote $\varepsilon_h$ the set of all faces, $\varepsilon_h^\partial$ the set of faces such that $e\subset\Gamma$, and $\varepsilon_h^0 = \varepsilon_h \setminus \varepsilon_h^\partial$.  The mesh dependent inner products are denoted by
\begin{align*}
(w,v)_{\mathcal{T}_h} = \sum_{K\in\mathcal{T}_h} (w,v)_K,   \quad\quad\quad\quad\left\langle \zeta,\rho\right\rangle_{\partial\mathcal{T}_h} = \sum_{K\in\mathcal{T}_h} \left\langle \zeta,\rho\right\rangle_{\partial K}.
\end{align*}

Let $\mathcal{P}^k(D)$ denote the set of polynomials of degree at most $k$ on a domain $D$.  As in Part I, we use the discontinuous finite element spaces
\begin{align}
\bm{V}_h  &:= \{\bm{v}\in [L^2(\Omega)]^d: \bm{v}|_{K}\in [\mathcal{P}^k(K)]^d, \forall K\in \mathcal{T}_h\},\\
{W}_h  &:= \{{w}\in L^2(\Omega): {w}|_{K}\in \mathcal{P}^{k+1}(K), \forall K\in \mathcal{T}_h\},\\
{M}_h  &:= \{{\mu}\in L^2(\mathcal{\varepsilon}_h): {\mu}|_{e}\in \mathcal{P}^{k+1}(e), \forall e\in \varepsilon_h\}
\end{align}
for the flux variables, scalar variables, and boundary trace variables, respectively.  Note that the polynomial degree for the scalar and boundary trace variables is one order higher than the polynomial degree for the flux variables.  We discussed this unusual choice for $ M_h $ in Part I.

Define $M_h(o)$ and $M_h(\partial)$ in the same way as $M_h$, but with $\varepsilon_h^o$ and $\varepsilon_h^{\partial}$ replacing $ \varepsilon_h $, respectively. For any functions $w\in W_h$ and $\bm r\in \bm V_h$, we use $\nabla w$ and $ \nabla \cdot \bm r $ to denote the gradient of $ w $ and the divergence of $ \bm r $ taken piecewise on each element $K\in \mathcal T_h$.

%\subsection{The HDG Formulation}

To approximate the solution of the mixed weak form \eqref{mixed_a}-\eqref{mixed_e} of the optimality system, the HDG formulation considered here is modified from Part I to avoid the estimation of $\bm q$ on the boundary.  In the 2D case, recall from \Cref{sec:regularity} that $ \bm q $ is not guaranteed to have a well-defined $ L^2 $ boundary trace since we consider a solution of the optimal control problem with low regularity.

The HDG method seeks approximate fluxes ${\bm{q}}_h,{\bm{p}}_h \in \bm{V}_h $, states $ y_h, z_h \in W_h $, interior element boundary traces $ \widehat{y}_h^o,\widehat{z}_h^o \in M_h(o) $, and  boundary control $ u_h \in M_h(\partial)$ satisfying
\begin{subequations}\label{HDG_discrete2}
	\begin{align}
	%%%%%%%%%%%%%%%%%
	(\bm q_h,\bm r_1)_{\mathcal T_h}-( y_h,\nabla\cdot\bm r_1)_{\mathcal T_h}+\langle \widehat y_h^o,\bm r_1\cdot\bm n\rangle_{\partial\mathcal T_h\backslash \varepsilon_h^\partial}+\langle  u_h,\bm r_1\cdot\bm n\rangle_{\varepsilon_h^\partial}&=0, \label{HDG_discrete2_a}\\
	%%%%%%%%%%%%%%%%%
	(\nabla\cdot\bm{q}_h,  w_1)_{{\mathcal{T}_h}} - (\bm{\beta} y_h, \nabla w_1)_{\mathcal T_h} - (\nabla\cdot\bm{\beta} y_h, w_1)_{\mathcal T_h} \quad &\nonumber\\
	%%%%%%%%%%%%%%
	\quad + \langle (h^{-1}+\tau_1)  y_h, w_1\rangle_{\partial\mathcal T_h}  +\langle  (\bm{\beta}\cdot\bm n - \tau_1 - h^{-1})\widehat y_h^o, w_1 \rangle_{\partial{{\mathcal{T}_h}}\backslash\varepsilon_h^\partial}  \quad &\nonumber\\ 
	%%%%%%%%%%%%%%%%%%%%%%%%
	+ \langle  (\bm{\beta}\cdot\bm n - \tau_1-h^{-1}) u_h, w_1 \rangle_{\varepsilon_h^\partial}&=(f, w_1)_{{\mathcal{T}_h}}, \label{HDG_discrete2_b}
	\end{align}
	for all $(\bm{r}_1, w_1)\in \bm{V}_h\times W_h$,
	\begin{align}
	%%%%%%%%%%%%%%%%%
	(\bm p_h,\bm r_2)_{\mathcal T_h}-(z_h,\nabla\cdot\bm r_2)_{\mathcal T_h}+\langle \widehat z_h^o,\bm r_2\cdot\bm n\rangle_{\partial\mathcal T_h\backslash\varepsilon_h^\partial}&=0,\label{HDG_discrete2_c}\\
	%%%%%%%%%%%%%%%%%
	(\nabla\cdot\bm{p}_h,  w_2)_{{\mathcal{T}_h}} - (y_h, w_2)_{\mathcal T_h} + (\bm{\beta} z_h, \nabla w_2)_{\mathcal T_h} \quad &\nonumber\\
	%%%%%%%%%%%%%%
	+ \langle (h^{-1}+\tau_2) z_h, w_2\rangle_{\partial \mathcal T_h} -\langle (h^{-1}+\tau_2 + \bm{\beta}\cdot \bm n) \widehat{z}_h^o, w_2 \rangle_{\partial{{\mathcal{T}_h}}\backslash \varepsilon_h^{\partial}} &=- (y_d, w_2)_{{\mathcal{T}_h}} ,  \label{HDG_discrete2_d}
	\end{align}
	for all $(\bm{r}_2, w_2)\in \bm{V}_h\times W_h$,
	\begin{align}
	\langle{\bm{q}_h}\cdot \bm{n}, \mu_1 \rangle_{\partial\mathcal{T}_{h}\backslash {\varepsilon_h^{\partial}}}+\langle (h^{-1}+\tau_1) y_h,\mu_1 \rangle_{\partial\mathcal{T}_{h}\backslash {\varepsilon_h^{\partial}}} \quad &\nonumber\\
	+\langle (\bm{\beta}\cdot\bm n - \tau_1-h^{-1})\widehat{y}_h^{o} ,\mu_1 \rangle_{\partial\mathcal{T}_{h}\backslash {\varepsilon_h^{\partial}}}&=0,\label{HDG_discrete2_e}
	\end{align}
	for all $\mu_1\in M_h(o)$,
	\begin{align}
	\langle {\bm{p}_h}\cdot \bm{n}, \mu_2\rangle_{ \partial\mathcal{T}_{h}\backslash {\varepsilon_h^{\partial}}}+\langle (h^{-1}+\tau_2) z_h, \mu_2\rangle_{ \partial\mathcal{T}_{h}\backslash {\varepsilon_h^{\partial}}} \quad &\nonumber\\
	-\langle (\bm{\beta}\cdot\bm n+\tau_2+h^{-1}) \widehat{z}_h^o, \mu_2\rangle_{ \partial\mathcal{T}_{h}\backslash {\varepsilon_h^{\partial}}} &=0,\label{HDG_discrete2_f}
	\end{align}
	for all $\mu_2\in M_h(o)$, and the optimality condition
	\begin{align}
	\langle {\bm{p}_h}\cdot \bm{n}, \mu_3\rangle_{{\varepsilon_h^{\partial}}} + \gamma\langle u_h, \mu_3 \rangle_{\varepsilon_h^{\partial}}+\langle (h^{-1}+\tau_2) z_h, \mu_3\rangle_{{\varepsilon_h^{\partial}}}&=0, \label{HDG_discrete2_g}
	\end{align}
	for all $\mu_3\in M_h(\partial)$.  
\end{subequations}

Here, $\tau_1$ and $\tau_2$ are stabilization functions defined on $\partial \mathcal T_h$ that satisfy the same conditions as in Part I:
\begin{description}
	
	\item[\textbf{(A1)}] $\tau_2$ is piecewise constant on $\partial \mathcal T_h$.
	
	\item[\textbf{(A2)}] $\tau_1 = \tau_2 + \bm{\beta}\cdot \bm n$.
	
	\item[\textbf{(A3)}] For any  $K\in\mathcal T_h$, $\min{(\tau_2+\frac 1 2 \bm \beta \cdot \bm n)}|_{\partial K} >0$.
	
\end{description}
Conditions \textbf{(A2)} and \textbf{(A3)} imply
\begin{equation}\label{eqn:tau1_condition}
\min{(\tau_1-\frac 1 2 \bm \beta \cdot \bm n)}|_{\partial K} >0  \quad  \mbox{for any $K\in\mathcal T_h$}.
\end{equation}
This completes the formulation of the HDG method.

Notice that formulation \eqref{HDG_discrete2} is slightly different from formulation (3.4) in Part~I; specifically, equations (b) and (d) are modified.  A straightforward computation shows that both are equivalent; see Part I, Section 3.2.  Formulation \eqref{HDG_discrete2} above allows us to achieve error estimates in the low regularity case considered here.

\section{Error Analysis}
\label{sec:analysis}

Next, we perform a convergence analysis of the above HDG method.

\subsection{Assumptions and Main Result}

As in Part I, we assume throughout that $\Omega$ is a bounded convex polyhedral domain and $ \bm \beta $ satisfies
\begin{equation}\label{eqn:beta_assumptions2}
\bm \beta \in [C(\overline{\Omega})]^d,  \quad  \nabla\cdot\bm{\beta}\in L^\infty(\Omega),  \quad  \nabla \cdot \bm \beta \leq 0,  \quad  \nabla \nabla \cdot \bm{\beta}\in[L^2(\Omega)]^d.
\end{equation}
We assume the solution of the optimality system \eqref{mixed_a}-\eqref{mixed_e} has the following regularity properties:
\begin{subequations}\label{eqn:regularity2}
	\begin{gather}
	y \in H^{r_y}(\Omega),  \quad  z \in H^{r_z}(\Omega),  \quad  \bm q \in [ H^{r_{\bm q}}(\Omega) ]^d \cap H(\mathrm{div},\Omega),  \quad  \bm p \in H^{r_{\bm p}}(\Omega),\\
	r_y \ge 1,  \quad  r_z \ge 2,  \quad  r_{\bm q} \ge  0,  \quad  r_{\bm p} \ge 1.\label{eqn:s_rates_ineq}
	\end{gather}
\end{subequations}
In the 2D case, \Cref{MT210} guarantees this regularity condition is satisfied.

As mentioned in \Cref{sec:regularity}, the regularity of $ \bm q $ can be low and therefore $ \bm q $ may not have a $ L^2 $ boundary trace.  The $ H(\mathrm{div},\Omega) $ regularity of $ \bm q $ is critically important for the numerical analysis.

We also require the family of meshes $ \{ \mathcal T_h \} $ is a conforming quasi-uniform triangulation of $ \Omega $.  This assumption on the meshes is stronger than in Part I; there we assumed $ \{ \mathcal T_h \} $ is a conforming quasi-uniform polyhedral mesh.  Therefore, the analysis in Part I allows for a more general family of meshes; however, the analysis here allows us to treat the low regularity case.

We now state our main convergence result.
\begin{theorem}\label{main_res}
	Let
	\begin{equation}\label{eqn:s_rates}
	\begin{split}
	s_{\bm q} &= \min\{r_{\bm q}, k+1\}, \qquad  s_{y} = \min\{r_{y}, k+2\}, \\
	s_{\bm p} &= \min\{r_{\bm p}, k+1\},  \qquad s_{z} = \min\{r_{z}, k+2\}.
	\end{split}
	\end{equation}
	If the above assumptions hold and $ s_{\bm q} \in [0,1] $, then
	\begin{align*}
	\norm{u-u_h}_{\varepsilon_h^\partial}&\lesssim h^{s_{\bm p}-\frac 1 2}\norm{\bm p}_{s_{\bm p},\Omega} +  h^{s_{z}-\frac 3 2}\norm{z}_{s_{z},\Omega} + h^{s_{\bm q}+\frac 1 2}\norm{\bm q}_{s_{\bm q},\Omega} + h^{s_{y}-\frac 12 }\norm{y}_{s_{y},\Omega},\\
	\norm {y-y_h}_{\mathcal T_h} &\lesssim h^{s_{\bm p}-\frac 1 2}\norm{\bm p}_{s_{\bm p},\Omega} +  h^{s_{z}-\frac 3 2}\norm{z}_{s_{z},\Omega} + h^{s_{\bm q}+\frac 1 2}\norm{\bm q}_{s_{\bm q},\Omega} + h^{s_{y}-\frac 12 }\norm{y}_{s_{y},\Omega},\\
	%\norm {\bm q - \bm q_h}_{\mathcal T_h} &\lesssim h^{s_{\bm p}-1}\norm{\bm p}_{s_{\bm p},\Omega} +  h^{s_{z}-2}\norm{z}_{s_{z},\Omega} + h^{s_{\bm q}}\norm{\bm q}_{s_{\bm q},\Omega} + h^{s_{y}-1 }\norm{y}_{s_{y},\Omega},\\
	\norm {\bm p - \bm p_h}_{\mathcal T_h}  &\lesssim h^{s_{\bm p}-\frac 1 2}\norm{\bm p}_{s_{\bm p},\Omega} +  h^{s_{z}-\frac 3 2}\norm{z}_{s_{z},\Omega} + h^{s_{\bm q}+\frac 1 2}\norm{\bm q}_{s_{\bm q},\Omega} + h^{s_{y}-\frac 12 }\norm{y}_{s_{y},\Omega},\\
	\norm {z - z_h}_{\mathcal T_h} & \lesssim  h^{s_{\bm p}-\frac 1 2}\norm{\bm p}_{s_{\bm p},\Omega} +  h^{s_{z}-\frac 3 2}\norm{z}_{s_{z},\Omega} + h^{s_{\bm q}+\frac 1 2}\norm{\bm q}_{s_{\bm q},\Omega} + h^{s_{y}-\frac 12 }\norm{y}_{s_{y},\Omega}.
	\end{align*}
	If in addition the inequalities in \eqref{eqn:s_rates_ineq} are strict and $ k \geq 1 $, then
	\begin{align*}
	\norm {\bm q - \bm q_h}_{\mathcal T_h} &\lesssim h^{s_{\bm p}-1}\norm{\bm p}_{s_{\bm p},\Omega} +  h^{s_{z}-2}\norm{z}_{s_{z},\Omega} + h^{s_{\bm q}}\norm{\bm q}_{s_{\bm q},\Omega} + h^{s_{y}-1 }\norm{y}_{s_{y},\Omega}.\\
	\end{align*}
\end{theorem}

\begin{remark}
	Note that we assume $ s_{\bm q} \in [0,1] $.  This is not a restriction since the case $ s_{\bm q} > 1 $ is treated in Part I on a more general family of meshes.
\end{remark}

Specializing to the 2D case gives the following result:
\begin{corollary}\label{cor:main_result}
	Suppose $ d = 2 $, $ f = 0 $, $ s_{\bm q} \in [0,1] $, and $ y_d \in H^{t^*}(\Omega) $ for some $ t^* \in [0,1) $.  Let $ \pi/3\le \omega<\pi $ be the largest interior angle of $\Gamma$, and let $ r > 0 $ satisfy
	$$
	r \leq  r_d := \frac{1}{2} + t^* \in [1/2,3/2),  \quad  \mbox{and}  \quad  r < r_{\Omega} := \min\left\{ \frac{3}{2}, \frac{\pi}{\omega} - \frac{1}{2} \right\} \in (1/2, 3/2].
	$$
	If $ k = 1 $, then
	\begin{align*}
	\norm{u-u_h}_{\varepsilon_h^\partial}&\lesssim h^{r} (\norm{\bm p}_{H^{r+1/2}(\Omega)} +  \norm{z}_{H^{r+3/2}(\Omega)} + \norm{\bm q}_{H^{r-1/2}(\Omega)} + \norm{y}_{H^{r+1/2}(\Omega)}),\\
	\norm{y-y_h}_{\mathcal T_h}&\lesssim h^{r} (\norm{\bm p}_{H^{r+1/2}(\Omega)} +  \norm{z}_{H^{r+3/2}(\Omega)} + \norm{\bm q}_{H^{r-1/2}(\Omega)} + \norm{y}_{H^{r+1/2}(\Omega)}),\\
	%
	%
	%
	%\norm {\bm q - \bm q_h}_{\mathcal T_h}  &\lesssim h^{r-1/2} (\norm{\bm p}_{H^{r+1/2}(\Omega)} +  \norm{z}_{H^{r+3/2}(\Omega)} + \norm{\bm q}_{H^{r-1/2}(\Omega)} + \norm{y}_{H^{r+1/2}(\Omega)}),\\
	%
	%
	%
	\norm {\bm p - \bm p_h}_{\mathcal T_h}   &\lesssim h^{r} (\norm{\bm p}_{H^{r+1/2}(\Omega)} +  \norm{z}_{H^{r+3/2}(\Omega)} + \norm{\bm q}_{H^{r-1/2}(\Omega)} + \norm{y}_{H^{r}+1/2(\Omega)}),\\
	\norm {z - z_h}_{\mathcal T_h}   & \lesssim  h^{r} (\norm{\bm p}_{H^{r+1/2}(\Omega)} +  \norm{z}_{H^{r+3/2}(\Omega)} + \norm{\bm q}_{H^{r-1/2}(\Omega)} + \norm{y}_{H^{r+1/2}(\Omega)}).
	\end{align*}
	If in addition $ r > 1/2 $, then
	$$
	\norm {\bm q - \bm q_h}_{\mathcal T_h}  \lesssim h^{r-1/2} (\norm{\bm p}_{H^{r+1/2}(\Omega)} +  \norm{z}_{H^{r+3/2}(\Omega)} + \norm{\bm q}_{H^{r-1/2}(\Omega)} + \norm{y}_{H^{r+1/2}(\Omega)}).
	$$
	Furthermore, if $ k = 0 $ then
	\begin{align*}
	\norm{u-u_h}_{\varepsilon_h^\partial}&\lesssim h^{1/2} (\norm{\bm p}_{H^{1}(\Omega)} +  \norm{z}_{H^{2}(\Omega)} + \norm{\bm q}_{H^{r-1/2}(\Omega)} + \norm{y}_{H^{r+1/2}(\Omega)}),\\
	\norm{y-y_h}_{\mathcal T_h}&\lesssim h^{1/2} (\norm{\bm p}_{H^{1}(\Omega)} +  \norm{z}_{H^{2}(\Omega)} + \norm{\bm q}_{H^{r-1/2}(\Omega)} + \norm{y}_{H^{r+1/2}(\Omega)}),\\
	\norm {\bm p - \bm p_h}_{\mathcal T_h}   &\lesssim h^{1/2} (\norm{\bm p}_{H^{1}(\Omega)} +  \norm{z}_{H^{2}(\Omega)} + \norm{\bm q}_{H^{r-1/2}(\Omega)} + \norm{y}_{H^{r+1/2}(\Omega)}),\\
	\norm {z - z_h}_{\mathcal T_h}   & \lesssim  h^{1/2} (\norm{\bm p}_{H^{1}(\Omega)} +  \norm{z}_{H^{2}(\Omega)} + \norm{\bm q}_{H^{r-1/2}(\Omega)} + \norm{y}_{H^{r+1/2}(\Omega)}).
	\end{align*}
	%
%		for any $ \varepsilon > 0 $.
\end{corollary}
As in Part I, when $ k = 1 $ the convergence rates are optimal for the control and the flux $ \bm q $ and suboptimal for the other variables.  When $ k = 0 $ the convergence rates for all variables are suboptimal with one exception:  If $ y_d \in L^2(\Omega) $ only so that $ t^* = 0 $, then $ u \in H^{1/2}(\Gamma) $ only and the convergence rate for the control is optimal.  Also, if $ r_d $ or $ r_\Omega $ is near $ 1/2 $, then the convergence rate is nearly optimal for the control in the $ k = 0 $ case.

\subsection{Preliminary material I}
\label{sec:Projectionoperator}

We split the preliminary material required for the proof into two parts.  First, we give a brief overview of material closely related to the preliminary material in Part I: $ L^2 $ projections, HDG operators $ \mathscr B_1 $ and $ \mathscr B_2 $, and the well-posedness of the HDG equations.

As in Part I, we use the standard $L^2$ projections $\bm\Pi :  [L^2(\Omega)]^d \to \bm V_h$, $\Pi :  L^2(\Omega) \to W_h$, and $P_M:  L^2(\varepsilon_h) \to M_h$, which satisfy
\begin{equation}\label{L2_projection}
\begin{split}
(\bm\Pi \bm q,\bm r)_{K} &= (\bm q,\bm r)_{K} ,\qquad \forall \bm r\in [{\mathcal P}_{k}(K)]^d,\\
(\Pi y,w)_{K}  &= (y,w)_{K} ,\qquad \forall w\in \mathcal P_{k+1}(K),\\
\left\langle P_M m, \mu\right\rangle_{ e} &= \left\langle  m, \mu\right\rangle_{e }, \quad\;\;\; \forall \mu\in \mathcal P_{k+1}(e).
\end{split}
\end{equation}
We have the following bounds:
\begin{subequations}\label{classical_ine}
	\begin{align}
	&\norm {\bm q -\bm{\Pi q}}_{\mathcal T_h} \lesssim h^{s_{\bm q}} \norm{\bm q}_{s_{\bm q},\Omega},\quad\quad\,\, \norm {y -{\Pi y}}_{\mathcal T_h} \lesssim h^{s_{y}} \norm{y}_{s_{y},\Omega},\\
	&\norm {y -{\Pi y}}_{\partial\mathcal T_h} \lesssim h^{s_{y}-\frac 1 2} \norm{y}_{s_{y},\Omega},
	\quad \norm {w}_{\partial \mathcal T_h} \lesssim h^{-\frac 12} \norm {w}_{ \mathcal T_h}, \: \forall w\in W_h,
	\end{align}
\end{subequations}
and similar projection error bounds for $\bm p$ and $z$.

In this paper, we do not use the same HDG formulation for the analysis that we used in Part I.  We define the HDG operators $ \mathscr B_1$ and $ \mathscr B_2 $ by
\begin{align}
\hspace{3em}&\hspace{-3em} \mathscr  B_1( \bm q_h,y_h,\widehat y_h^o;\bm r_1,w_1,\mu_1) \nonumber\\
& = (\bm q_h,\bm r_1)_{\mathcal T_h}-( y_h,\nabla\cdot\bm r_1)_{\mathcal T_h}+\langle \widehat y_h^o,\bm r_1\cdot\bm n\rangle_{\partial\mathcal T_h\backslash \varepsilon_h^\partial}+(\nabla\cdot \bm q_h,w_1)_{\mathcal{T}_h} \nonumber\\ 
& \quad -(\bm \beta y_h,  \nabla w_1)_{\mathcal T_h}-(\nabla\cdot\bm\beta y_h,w_1)_{\mathcal T_h}+\langle h^{-1}y_h +\tau_1 y_h,w_1\rangle_{\partial\mathcal T_h}\nonumber\\
& \quad +\langle (\bm\beta\cdot\bm n -h^{-1}-\tau_1) \widehat y_h^o,w_1\rangle_{\partial\mathcal T_h\backslash \varepsilon_h^\partial}\nonumber\\
&\quad -\langle \bm q_h\cdot \bm n+\bm \beta\cdot\bm n\widehat y_h^o +h^{-1}(y_h-\widehat y_h^o) + \tau_1(y_h - \widehat y_h^o),\mu_1\rangle_{\partial\mathcal T_h\backslash\varepsilon^{\partial}_h},\label{def_B1}\\
\hspace{3em}&\hspace{-3em} \mathscr B_2 (\bm p_h,z_h,\widehat z_h^o;\bm r_2, w_2,\mu_2)\nonumber\\
&=(\bm p_h,\bm r_2)_{\mathcal T_h}-( z_h,\nabla\cdot\bm r_2)_{\mathcal T_h}+\langle \widehat z_h^o,\bm r_2\cdot\bm n\rangle_{\partial\mathcal T_h\backslash\varepsilon_h^\partial}+(\nabla\cdot\bm p_h,w_2)\nonumber\\
& \quad +(\bm \beta z_h,  \nabla w_2)_{\mathcal T_h}+\langle h^{-1} z_h +\tau_2 z_h ,w_2\rangle_{\partial\mathcal T_h} \nonumber\\
&\quad-\langle (\bm \beta\cdot\bm n + h^{-1}+\tau_2)\widehat z_h^o ,w_2\rangle_{\partial\mathcal T_h\backslash\varepsilon_h^\partial}\nonumber\\
& \quad -\langle  {\bm p}_h\cdot\bm n-\bm \beta\cdot\bm n\widehat z_h^o +h^{-1}(z_h-\widehat z_h^o)+\tau_2(z_h - \widehat z_h^o),\mu_2\rangle_{\partial\mathcal T_h\backslash\varepsilon^{\partial}_h}\label{def_B2}.
\end{align}
We emphasize that this is an equivalent definition to the one given in Part I that is more appropriate to obtain error estimates in the low regularity case.

We rewrite the HDG formulation of the optimality system \eqref{HDG_discrete2} in terms of the HDG operators $\mathscr B_1$ and $\mathscr B_2$: find $({\bm{q}}_h,{\bm{p}}_h,y_h,z_h,\widehat y_h^o,\widehat z_h^o,u_h)\in \bm{V}_h\times\bm{V}_h\times W_h \times W_h\times M_h(o)\times M_h(o)\times M_h(\partial)$ satisfying
\begin{subequations}\label{HDG_full_discrete}
	\begin{align}
	%%%%%%%%%%%%%%%%%%%
	\mathscr B_1(\bm q_h,y_h,\widehat y_h^o;\bm r_1,w_1,\mu_1)&=( f, w_1)_{\mathcal T_h} - \langle u_h,\bm r_1\cdot\bm n\rangle_{\varepsilon_h^\partial} \nonumber \\
	%%%%%%%%%%%%%%%%%%%
	&\quad-\langle (\bm\beta\cdot\bm n-h^{-1}-\tau_1) u_h, w_1\rangle_{\varepsilon_h^\partial},\label{HDG_full_discrete_a}\\
	%%%%%%%%%%%%%%%%%%%
	\mathscr B_2(\bm p_h,z_h,\widehat z_h^o;\bm r_2,w_2,\mu_2)&=(y_h-y_d,w_2)_{\mathcal T_h},\label{HDG_full_discrete_b}\\
	%%%%%%%%%%%%%%%%%%%
	\gamma^{-1}\langle {\bm{p}}_h\cdot \bm{n} + h^{-1} z_h + \tau_2 z_h, \mu_3\rangle_{{{\varepsilon_h^{\partial}}}} &= -\langle u_h, \mu_3 \rangle_{{\varepsilon_h^{\partial}}}, \label{HDG_full_discrete_e}
	\end{align}
\end{subequations}
for all $\left(\bm{r}_1, \bm{r}_2, w_1, w_2, \mu_1, \mu_2, {\mu}_3\right)\in \bm V_h\times\bm V_h\times W_h\times W_h\times M_h(o)\times M_h(o)\times M_h(\partial) $.

For the convenience of the reader, we recall three results proven in Part I.
\begin{lemma}\label{property_B}
	For any $ ( \bm v_h, w_h, \mu_h ) \in \bm V_h \times W_h \times M_h $, we have
	\begin{align*}
	%%%%%%%%%%%%%%%%
	\hspace{2em}&\hspace{-2em} \mathscr B_1(\bm v_h,w_h,\mu_h;\bm v_h,w_h,\mu_h)\\
	%%%%%%%%%%%%%%%
	&=(\bm v_h,\bm v_h)_{\mathcal T_h}+ \langle (h^{-1}+\tau_1 - \frac 12 \bm \beta\cdot\bm n)(w_h-\mu_h),w_h-\mu_h\rangle_{\partial\mathcal T_h\backslash \varepsilon_h^\partial}\\
	%%%%%%%%%%%%%%%%%%%
	&\quad-\frac 1 2(\nabla\cdot\bm\beta w_h,w_h)_{\mathcal T_h} +\langle (h^{-1}+\tau_1-\frac12\bm \beta\cdot\bm n) w_h,w_h\rangle_{\varepsilon_h^\partial},\\
	%%%%%%%%%%%%%%%%%%%%%%%
	\hspace{2em}&\hspace{-2em}\mathscr B_2(\bm v_h,w_h,\mu_h;\bm v_h,w_h,\mu_h)\\
	%%%%%%%%%%%%%%%%%%%%%%%
	&=(\bm v_h,\bm v_h)_{\mathcal T_h}+ \langle (h^{-1}+\tau_2 + \frac 12 \bm \beta\cdot\bm n)(w_h-\mu_h),w_h-\mu_h\rangle_{\partial\mathcal T_h\backslash \varepsilon_h^\partial}\\
	%%%%%%%%%%%%%%%%%%%%%%%%
	&\quad-\frac 1 2(\nabla\cdot\bm\beta w_h,w_h)_{\mathcal T_h} +\langle (h^{-1}+\tau_2+\frac12\bm \beta\cdot\bm n) w_h,w_h\rangle_{\varepsilon_h^\partial}.
	\end{align*}
\end{lemma}

\begin{lemma}\label{identical_equa}
	If \textbf{(A2)} holds, then
	$$\mathscr B_1 (\bm q_h,y_h,\widehat y_h^o;\bm p_h,-z_h,-\widehat z_h^o) + \mathscr B_2 (\bm p_h,z_h,\widehat z_h^o;-\bm q_h,y_h,\widehat y_h^o) = 0.$$
\end{lemma}

\begin{proposition}\label{ex_uni}
	If \textbf{(A2)} holds, there exists a unique solution of the HDG equations \eqref{HDG_full_discrete}.
\end{proposition}

\subsection{Preliminary material II}
\label{sec:Projectionoperator2}
Next, we discuss preliminary material that is directly related to the low regularity case considered in this paper: the interpolation operators $\mathcal I_h^0$, $\mathcal I_h^1$,  $\mathcal I_h$ and their properties.

Recall we assume the primary flux $ \bm q $ only satisfies $ \bm q \in [H^{r_{\bm q}}(\Omega)]^d \cap H(\mathrm{div},\Omega) $, where $r_{\bm q}\ge 0$.  Therefore, the quantity $\norm {\bm q \cdot \bm n -\bm{\Pi q}\cdot\bm n}_{\partial\mathcal T_h}$ is not well defined and the HDG analysis technique used in Part I is not applicable.  We use analysis techniques from \cite{MR3508837,LiXie16_SINUM} to avoid using the $ L^2 $ boundary trace of $ \bm q $. Let us  introduce some notation first.

Define the $H^1$-conforming piecewise linear finite element space $W_h^c$ by
\begin{align*}
W_h^c:= \{ w_h^c\in H_0^1(\Omega):w_h^c|_K\in \mathcal{P}_1(K), \  \forall K\in \mathcal{T}_h \} .
\end{align*}
For any $K\in \mathcal{T}_h$, let $\lambda_1$, $\lambda_2$, $\ldots$, $\lambda_{d+1}$ denote the standard barycentric coordinate functions defined on the simplex $K$. Define
\begin{align} \label{def_Gamma}
\mathbb S(K):=S_1(K)+S_2(K)+\cdots+S_{d+1}(K),
\end{align}
where
\begin{align*}
S_i(K): =\bigg( \prod_{j \neq i} \lambda_j  \bigg) \text{span} \bigg\{ \prod_j \lambda_j^{\alpha_j}:\sum_j \alpha_j=k,~\alpha_i =0  \bigg\}, \ \  i=1,2,\ldots,d+1.
\end{align*}

Now we define the interpolations operators $\mathcal I_h^0$, $\mathcal I_h^1$,  $\mathcal I_h$.  First, define $m_K: L^2(\partial K)\rightarrow \mathbb{R}$ by
\begin{align}\label{mean_mu}
m_K(\mu):=\frac{1}{d+1}\sum_{e\in \partial K} \frac{1}{|e|}\int_e \mu,%\quad\forall \mu\in L^2(\partial K), 
\end{align}
where $|e|$ denotes the $d-1$ dimensional Hausdorff measure of $ e $.  Next, the interpolation operator $\mathcal I_h^0:L^2(\varepsilon_h)\rightarrow W_h^c$ is defined as follows:
\begin{align*}
\mathcal I_h^0 \mu(a)  = 
\begin{cases}
\displaystyle{ \frac{1}{\# \omega_a} \sum_{K\in \omega_a}m_K(\mu) } & \mbox{if $a$ is an interior node of $\mathcal{T}_h$,} \\
\qquad\qquad  0  &  \mbox{if $a$ is a boundary node of $ \mathcal{T}_h $,}
\end{cases}
\end{align*}
where $\omega_a:=\{ K\in\mathcal{T}_h:a~\text{is a vertex of }K \}$ and $\# \omega_a$ denotes the number of elements in $\omega_a$.

Next, the interpolation operator $\mathcal I_h^1$ on $L^2(\Omega)\times L^2(\varepsilon_h)$ is defined elementwise as follows: for each $ K $, 
\begin{align*}
\mathcal I_h^1(w,\mu)|_K :=\mathcal I_K^1(w,\mu) = w_1 + w_2,
\end{align*}
where $(w_1,w_2)\in \mathbb S(K)\times  (\prod_j \lambda_j)\mathcal{P}_k(K)$ is uniquely determined by
\begin{align*}
\langle w_1,m\rangle_e &=\langle \mu,m\rangle_e,\\
(w_2,n)_K &= (w-w_1, n)_K,
\end{align*}
for all $(m,n)\in\mathcal P_k(e)\times \mathcal P_k(K)$ and $e\in\partial K$.

Finally, for $(w,\mu)\in L^2(\Omega)\times L^2(\varepsilon_h)$,  we define the third interpolation operator $ \mathcal I_h $ by
\begin{align*}
\mathcal I_h(w,\mu) :=\mathcal I_h^0\mu +\mathcal I_h^1(w-\mathcal I_h^0 \mu,\mu -\mathcal I_h^0 \mu).
\end{align*}
It is straightforward to verify that $\mathcal I_h$ and $\mathcal I_h^1$ have the following  properties; see \cite{MR3508837,LiXie16_SINUM}.
\begin{lemma} \label{interpolation_projection}
	For any $(w,\mu)\in L^2(\Omega)\times L^2(\varepsilon_h)$ and $K\in \mathcal{T}_h$, we have
	\begin{subequations}
		\begin{align}
		(\mathcal I_h(w,\mu),n)_K&=(w,n)_K,\label{interpolation_projection1}\\
		\langle \mathcal I_h (w,\mu),m \rangle_{\partial K}&=\langle \mu,m \rangle_{\partial K},\label{interpolation_projection2}
		\end{align}
	\end{subequations}
	for all $(m,n)\in\mathcal P_k(e)\times \mathcal P_k(K)$ and $e\in\partial K$, and 
	\begin{align}\label{basic_err_est}
	\|\mathcal I_h^1(w,\mu)\|_K &\lesssim \|w\|_{K} + h^{\frac 1 2} \|\mu\|_{\partial K}.
	\end{align}
	Moreover, if $\mu|_{\Gamma} = 0$,  we have 
	\begin{align}\label{H01}
	\mathcal I_h(w,\mu) \in H_0^1(\Omega).
	\end{align}
\end{lemma}

In the next three lemmas, we assume $(\bm v_h,w_h,\mu_h)\in \bm V_h\times W_h\times M_h$ satisfy 
\begin{align}\label{first_error_eq}
(\bm v_h, \bm r)_{\mathcal{T}_h}-(w_h,\nabla\cdot \bm r)_{\mathcal{T}_h}+\langle \mu_h,\bm r\cdot \bm n \rangle_{\partial \mathcal{T}_h}=0,
\end{align}
for all $\bm r\in \bm V_h$.

We begin with a key inequality; see Part I \cite[Lemma 4.7]{HuMateosSinglerZhangZhang1} and also \cite{MR3440284}.
\begin{lemma}\label{nabla_ine}
	If $(\bm v_h,w_h,\mu_h)\in \bm V_h\times W_h\times M_h$ satisfy \eqref{first_error_eq}, then
	\begin{align} \label{error_grad}
	\| \nabla w_h \|_{\mathcal{T}_h} \lesssim \| \bm v_h \|_{\mathcal{T}_h} +h^{-\frac{1}{2}}\| w_h-\mu_h \|_{\partial \mathcal{T}_h}.
	\end{align}
\end{lemma}

The next two results are similar to Lemma 3.4 and Lemma 3.6 in \cite{MR3508837}.  Here, we have a different space $ M_h $ (with polynomials of degree $ k+1 $ instead of $ k $) and we do not have a variable diffusion coefficient.  However, the proofs of the next two results are very similar to the proofs in \cite{MR3508837} and are omitted.
\begin{lemma} \label{error_auxilliary}
	If $(\bm v_h,w_h,\mu_h)\in \bm V_h\times W_h\times M_h$ satisfy \eqref{first_error_eq}, then
	\begin{align}
	\hspace{3em}&\hspace{-3em}  h^{-1} \sum_{K\in\mathcal T_h}\| w_h-m_K(\mu_h) \|_{K} + h^{-\frac 1 2} \sum_{K\in\mathcal T_h} \| \mu_h-m_K(\mu_h) \|_{\partial K }  \nonumber\\
	&\lesssim \| \bm v_h \|_{\mathcal T_h}+h^{-\frac{1}{2}}\| w_h-\mu_h \|_{\partial \mathcal T_h}.
	\end{align}
	%\begin{multline}
	%	h^{-1} \sum_{K\in\mathcal T_h}\| w_h-m_K(\mu_h) \|_{K} + h^{-\frac 1 2} \sum_{K\in\mathcal T_h} \| \mu_h-m_K(\mu_h) \|_{\partial K }  \nonumber\\
	%	\lesssim \| \bm v_h \|_{\mathcal T_h}+h^{-\frac{1}{2}}\| w_h-\mu_h \|_{\partial \mathcal T_h}.
	%\end{multline}
\end{lemma}
\begin{lemma}  \label{error_esti_inter}
	If $(\bm v_h,w_h,\mu_h)\in \bm V_h\times W_h\times M_h$ satisfy \eqref{first_error_eq}, then
	\begin{subequations}
		\begin{align}
		\|\nabla\mathcal I_h(w_h,\mu_h)\|_{\mathcal T_h} &\lesssim \|\bm v_h\|_{\mathcal{T}_h} +h^{-\frac{1}{2}}\| w_h-\mu_h\|_{\partial \mathcal{T}_h}, \label{H_1_error} \\
		h^{-1} \|w_h-\mathcal I_h (w_h,\mu_h)\|_{\mathcal T_h}  &\lesssim \|\bm v_h\|_{\mathcal{T}_h} +h^{-\frac{1}{2}}\| w_h-\mu_h\|_{\partial \mathcal{T}_h}. \label{L_2_error}
		\end{align}
	\end{subequations}
\end{lemma}

\subsection{Proof of Main Result}
Now we move to the proof of the error estimates.  We follow the strategy of Part I \cite{HuMateosSinglerZhangZhang1} and split the proof into
seven steps.  In the first five steps we use the rewriting of operators $\mathscr B_1$ and $\mathscr B_2$ in an explicit way and the proofs are different from the corresponding ones of Part I.  Steps 6 and 7 use the properties of $\mathscr B_1$ and $\mathscr B_2$ recalled in \Cref{property_B} and \Cref{identical_equa} and are very similar to Steps 6 and 7 in the high regularity case in Part I.  We include these proofs here to make this paper self-contained.

We first bound the error between the solution of the mixed form \eqref{mixed_a}-\eqref{mixed_d} of the optimality system and the solution
$$
({\bm{q}}_h(u),{\bm{p}}_h(u), y_h(u), z_h(u), {\widehat{y}}_h^o(u), {\widehat{z}}_h^o(u))\in \bm{V}_h\times\bm{V}_h\times W_h \times W_h\times M_h(o)\times M_h(o)
$$
of the auxiliary problem
\begin{subequations}\label{HDG_inter_u}
	\begin{align}
	\mathscr B_1(\bm q_h(u), y_h(u), \widehat{y}^o_h(u); \bm r_1, w_1, \mu_1)&=( f, w_1)_{\mathcal T_h} - \langle P_M u, \bm r_1\cdot\bm n\rangle_{\varepsilon_h^\partial} \nonumber \\
	& \quad -\langle (\bm\beta\cdot\bm n-h^{-1}-\tau_1) P_M u, w_1\rangle_{\varepsilon_h^\partial},\label{HDG_u_a}\\
	\mathscr B_2(\bm p_h(u), z_h(u), \widehat{z}^o_h(u); \bm r_2,  w_2, \mu_2)&=(y_h(u) - y_d, w_2)_{\mathcal T_h},\label{HDG_u_b}
	\end{align}
\end{subequations}
for all $\left(\bm{r}_1, \bm{r}_2,w_1,w_2,\mu_1,\mu_2\right)\in \bm{V}_h\times\bm{V}_h \times W_h\times W_h\times M_h(o)\times M_h(o)$.  As in Part~I, we use the notation
\begin{align}\label{notation_1}
\begin{aligned}%column 1
\delta^{\bm q} &=\bm q-{\bm\Pi}\bm q,\\
\delta^y&=y- \Pi y,\\
\delta^{\widehat y} &= y-P_M y,\\
\widehat {\bm\delta}_1 &= \bm{\beta}\cdot\bm n \delta^{\widehat y} + (h^{-1}+\tau_1)(\delta^y - \delta^{\widehat y}),
\end{aligned}
&&
\begin{aligned}%column 2
\varepsilon^{\bm q}_h &= {\bm\Pi} \bm q-\bm q_h(u),\\
\varepsilon^{y}_h &= \Pi y-y_h(u),\\
\varepsilon^{\widehat y}_h &= P_M y-\widehat{y}_h(u),\\
\ & \
\end{aligned}
\end{align}
where $\widehat y_h(u) = \widehat y_h^o(u)$ on $\varepsilon_h^o$ and $\widehat y_h(u) = P_M u$ on $\varepsilon_h^{\partial}$.  This definition gives $\varepsilon_h^{\widehat y} = 0$ on $\varepsilon_h^{\partial}$.

\subsubsection{Step 1: The error equation for part 1 of the auxiliary problem \eqref{HDG_u_a}}
\label{subsec:proof_step1}

\begin{lemma}\label{lemma:step1_first_lemma}
	We have
	\begin{align}\label{error_equation_L2k1}
	\hspace{1em}&\hspace{-1em} \mathscr B_1 (\varepsilon_h^{\bm q},\varepsilon_h^{ y}, \varepsilon_h^{\widehat y}, \bm r_1, w_1, \mu_1)\nonumber\\
	&=-(\nabla\cdot\delta^{\bm q},w_1)_{\mathcal{T}_h}-\langle \bm \Pi \bm q\cdot\bm n,\mu_1 \rangle_{\mathcal{T}_h\backslash\varepsilon_h^\partial} +( \bm \beta \delta^y, \nabla w_1)_{{\mathcal{T}_h}}\nonumber\\
	& \quad +(\nabla\cdot\bm{\beta}\delta^y,w_1)_{\mathcal T_h} - \langle \widehat{\bm \delta}_1, w_1 \rangle_{\partial{{\mathcal{T}_h}}} + \langle \widehat{\bm \delta}_1, \mu_1 \rangle_{\partial{{\mathcal{T}_h}}\backslash \varepsilon_h^{\partial}}.
	\end{align}
\end{lemma}
\begin{proof}
	Using the definition of $ \mathscr B_1 $ in \eqref{def_B1} gives
	\begin{align*}
	\hspace{1em}&\hspace{-1em}  \mathscr B_1 (\bm \Pi {\bm q},\Pi { y}, P_M  y, \bm r_1, w_1, \mu_1)\\
	&= (\bm \Pi {\bm q}, \bm{r_1})_{{\mathcal{T}_h}}- (\Pi { y}, \nabla\cdot \bm{r_1})_{{\mathcal{T}_h}}+\langle P_M  y, \bm{r_1}\cdot \bm{n} \rangle_{\partial{{\mathcal{T}_h}}\backslash {\varepsilon_h^{\partial}}}\\
	&  \quad+(\nabla\cdot\bm \Pi {\bm q},w_1)_{\mathcal{T}_h}- ( \bm{\beta} \Pi y, \nabla w_1)_{{\mathcal{T}_h}}
	- (\nabla\cdot\bm{\beta} \Pi y,  w_1)_{{\mathcal{T}_h}}\\
	&\quad+\langle (h^{-1} + \tau_1) \Pi { y}, w_1 \rangle_{\partial{{\mathcal{T}_h}}}
	+ (\bm{\beta}\cdot\bm n-h^{-1}-\tau_1) P_M  y, w_1 \rangle_{\partial{{\mathcal{T}_h}}\backslash \varepsilon_h^{\partial}}\\
	&\quad	-\langle  \bm \Pi \bm q\cdot\bm n+\bm \beta\cdot\bm n P_M y +(h^{-1}+\tau_1)(\Pi y - P_M y),\mu_1\rangle_{\partial\mathcal T_h\backslash\varepsilon^{\partial}_h}.
	\end{align*}
	Using properties of the $ L^2 $ projections \eqref{L2_projection} gives
	\begin{align*}
	\hspace{3em}&\hspace{-3em} \mathscr B_1 (\bm \Pi {\bm q},\Pi { y},P_M  y, \bm r_1, w_1, \mu_1) \\
	&= ( {\bm q}, \bm{r_1})_{{\mathcal{T}_h}}- ({ y}, \nabla\cdot \bm{r_1})_{{\mathcal{T}_h}}+\langle   y, \bm{r_1}\cdot \bm{n} \rangle_{\partial{{\mathcal{T}_h}}\backslash {\varepsilon_h^{\partial}}}\\
	& \quad +(\nabla\cdot\bm q,w_1)_{\mathcal{T}_h}-(\nabla\cdot \delta^{\bm q},w_1)_{\mathcal{T}_h}- ( \bm \beta y, \nabla w_1)_{{\mathcal{T}_h}} +  (  \bm \beta \delta^y, \nabla w_1)_{{\mathcal{T}_h}} \\
	&\quad - (\nabla\cdot\bm{\beta} y, w_1)_{\mathcal T_h} + (\nabla\cdot\bm{\beta} \delta^y, w_1)_{\mathcal T_h}+\langle (h^{-1}+\tau_1)\Pi y , w_1 \rangle_{\partial{{\mathcal{T}_h}}}  \\
	&\quad+\langle\bm\beta\cdot\bm n y, w_1\rangle_{\partial\mathcal T_h\backslash\varepsilon_h^\partial} - \langle\bm\beta\cdot\bm n \delta^{\widehat y}, w_1\rangle_{\partial\mathcal T_h\backslash\varepsilon_h^\partial}
	- \langle (h^{-1}+\tau_1) P_M  y, w_1 \rangle_{\partial{{\mathcal{T}_h}}\backslash \varepsilon_h^{\partial}} \\
	&\quad- \langle \bm \Pi {\bm q}\cdot\bm n, \mu_1 \rangle_{\partial{{\mathcal{T}_h}}\backslash\varepsilon_h^{\partial}} - \langle {\bm \beta}\cdot\bm n y, \mu_1 \rangle_{\partial{{\mathcal{T}_h}}\backslash\varepsilon_h^{\partial}}+\langle {\bm \beta}\cdot\bm n \delta^{\widehat y}, \mu_1 \rangle_{\partial{{\mathcal{T}_h}}\backslash\varepsilon_h^{\partial}}\\
	&\quad+\langle (h^{-1}+\tau_1) (\delta^y-\delta^{\widehat y}), \mu_1 \rangle_{\partial{{\mathcal{T}_h}}\backslash\varepsilon_h^{\partial}}.
	\end{align*}
	The exact state $ y $ and flux $\bm{q}$ satisfy
	\begin{align*}
	(\bm{q},\bm{r}_1)_{\mathcal{T}_h}-(y,\nabla\cdot \bm{r}_1)_{\mathcal{T}_h}+\left\langle{y},\bm r_1\cdot \bm n \right\rangle_{\partial {\mathcal{T}_h}\backslash\varepsilon_h^\partial} &= -\langle u,\bm r_1\cdot \bm n \rangle_{\varepsilon_h^\partial},\\
	(\nabla\cdot \bm{q},w_1)_{\mathcal{T}_h}-(\bm{\beta} y,\nabla w_1)_{\mathcal{T}_h}-(\nabla \cdot \bm{\beta} y, w_1)_{\mathcal{T}_h}\\
	+\left\langle \bm \beta\cdot \bm n y,w_1\right\rangle_{\partial {\mathcal{T}_h}\backslash \varepsilon_h^\partial} &= -\langle \bm \beta \cdot \bm n u,w_1 \rangle_{\varepsilon_h^\partial}+(f,w_1)_{\mathcal{T}_h},
	\end{align*}
	for all $(\bm{r}_1,w_1)\in\bm{V}_h\times W_h$. This gives
	\begin{align*}
	\hspace{1em}&\hspace{-1em}  \mathscr B_1 (\bm \Pi {\bm q},\Pi { y}, P_M  y, \bm r_1, w_1, \mu_1)\\
	&=-\left\langle u,\bm r_1\cdot \bm n \right\rangle_{\varepsilon_h^{\partial}} - \left\langle \bm{\beta}\cdot \bm n u,w_1\right\rangle_{\varepsilon_h^{\partial}} + (f,w_1)_{\mathcal T_h}-(\nabla\cdot \delta^{\bm q},w_1)_{\mathcal{T}_h} +  ( \bm \beta \delta^y, \nabla w_1)_{{\mathcal{T}_h}} \\
	&\quad+(\nabla\cdot\bm{\beta}\delta^y,w_1)_{\mathcal T_h} +\langle (h^{-1} +\tau_1) \Pi y, w_1 \rangle_{\partial{{\mathcal{T}_h}}} \\
	&\quad- \langle {\bm \beta}\cdot\bm n \delta^{\widehat y}, w_1 \rangle_{\partial{{\mathcal{T}_h}}\backslash\varepsilon_h^{\partial}}  - \langle (h^{-1}+\tau_1) P_M  y, w_1 \rangle_{\partial{{\mathcal{T}_h}}\backslash \varepsilon_h^{\partial}} - \langle \bm \Pi \bm q\cdot\bm n, \mu_1 \rangle_{\partial{{\mathcal{T}_h}}\backslash\varepsilon_h^{\partial}}\\
	&\quad+\langle {\bm \beta}\cdot\bm n \delta^{\widehat y}, \mu_1 \rangle_{\partial{{\mathcal{T}_h}}\backslash\varepsilon_h^{\partial}}+ \langle (h^{-1}+\tau_1) (\delta^y-\delta^{\widehat y}), \mu_1 \rangle_{\partial{{\mathcal{T}_h}}\backslash\varepsilon_h^{\partial}}.
	\end{align*}
	Here we used $ \langle {\bm \beta}\cdot\bm n y, \mu_1 \rangle_{\partial{{\mathcal{T}_h}}\backslash\varepsilon_h^{\partial}} =0$, which holds since $\mu_1$ is a single-valued function on the interior edges. Subtracting part 1 of the auxiliary problem \eqref{HDG_u_a} from the above equality gives the result:
	\begin{align*}
	\hspace{1em}&\hspace{-1em} \mathscr B_1 (\varepsilon_h^{\bm q},\varepsilon_h^{ y}, \varepsilon_h^{\widehat y},\bm r_1, w_1, \mu_1) \\
	& =  - (\nabla\cdot \delta^{\bm q},w_1)_{\mathcal{T}_h} +( \bm \beta \delta^y, \nabla w_1)_{{\mathcal{T}_h}}
	+(\nabla\cdot\bm{\beta}\delta^y,w_1)_{\mathcal T_h}  \\
	&  \quad+\langle (h^{-1}+\tau_1) \Pi y, w_1 \rangle_{\partial{{\mathcal{T}_h}}}- \langle {\bm \beta}\cdot\bm n \delta^{\widehat y}, w_1 \rangle_{\partial{{\mathcal{T}_h}}}  - \langle (h^{-1}+\tau_1) P_M  y, w_1 \rangle_{\partial{{\mathcal{T}_h}}} \\
	&\quad- \langle \bm \Pi \bm q\cdot\bm n, \mu_1 \rangle_{\partial{{\mathcal{T}_h}}\backslash\varepsilon_h^{\partial}}+\langle {\bm \beta}\cdot\bm n \delta^{\widehat y}, \mu_1 \rangle_{\partial{{\mathcal{T}_h}}\backslash\varepsilon_h^{\partial}}+ \langle (h^{-1}+\tau_1) (\delta^y-\delta^{\widehat y}), \mu_1 \rangle_{\partial{{\mathcal{T}_h}}\backslash\varepsilon_h^{\partial}}\\
	&= -(\nabla\cdot \delta^{\bm q},w_1)_{\mathcal{T}_h}  + ( \bm \beta \delta^y, \nabla w_1)_{{\mathcal{T}_h}}
	+(\nabla\cdot\bm{\beta}\delta^y,w_1)_{\mathcal T_h}\\
	&\quad -\langle \bm \Pi \bm q\cdot \bm n,\mu_1 \rangle_{\partial \mathcal{T}_h\backslash \varepsilon_h^\partial} - \langle \widehat{\bm \delta}_1, w_1 \rangle_{\partial{{\mathcal{T}_h}}} + \langle \widehat{\bm \delta}_1, \mu_1 \rangle_{\partial{{\mathcal{T}_h}}\backslash \varepsilon_h^{\partial}}.
	\end{align*}
\end{proof}

\subsubsection{Step 2: Estimate for $\varepsilon_h^{\boldmath q}$}
\begin{lemma}\label{lemma:step2_main_lemma}
	We have
	\begin{align}
	\|\varepsilon_h^{\bm{q}}\|_{\mathcal{T}_h}+h^{-\frac 1 2}\|{\varepsilon_h^y-\varepsilon_h^{\widehat{y}}}\|_{\partial \mathcal T_h}  \lesssim h^{s_{\bm q}}\norm{\bm q}_{s^{\bm q},\Omega} + h^{s_{y}-1}\norm{y}_{s^{y},\Omega}.
	\end{align}
\end{lemma}

\begin{proof}
	First, take $ (\bm v_h,w_h,\mu_h) = (\varepsilon_h^{\bm q},\varepsilon_h^{y}, \varepsilon_h^{\widehat y}) $ in the key inequality in \Cref{nabla_ine} to obtain
	\begin{align}\label{nabla_y}
	\|\nabla \varepsilon_h^y\|_{\mathcal T_h}  \lesssim  \|\varepsilon^{\bm q}_h\|_{\mathcal T_h}+h^{-\frac1 2}\|\varepsilon^y_h-\varepsilon^{\widehat y}_h\|_{\partial\mathcal T_h}.
	\end{align}
	Next, since $\varepsilon_h^{\widehat y}=0$ on $\varepsilon_h^\partial$, the energy identity for $ \mathscr B_1 $ in \Cref{property_B} gives
	\begin{align*}
	\hspace{1em}&\hspace{-1em} \mathscr B(\varepsilon_h^{\bm q},\varepsilon_h^{ y}, \varepsilon_h^{\widehat y}, \varepsilon_h^{\bm q},\varepsilon_h^{ y}, \varepsilon_h^{\widehat y})\\ &=(\varepsilon_h^{\bm{q}},\varepsilon_h^{\bm{q}})_{\mathcal{T}_h}+ \|(h^{-1}+\tau_1-\frac 12 \bm{\beta} \cdot\bm n)^{\frac 12 } (\varepsilon_h^y-\varepsilon_h^{\widehat{y}})\|_{\partial\mathcal T_h}^2 +\frac 1 2\| (-\nabla\cdot\bm{\beta})^{\frac  1 2}\varepsilon_h^y\|_{\mathcal T_h}^2.
	\end{align*}
	Take $(\bm r_1, w_1,\mu_1) = (\bm \varepsilon_h^{\bm q},\varepsilon_h^y,\varepsilon_h^{\widehat y})$ in the error equation \eqref{error_equation_L2k1} in \Cref{lemma:step1_first_lemma} to obtain
	\begin{equation}\label{step2_2}
	\begin{split}
	(\varepsilon_h^{\bm{q}},\varepsilon_h^{\bm{q}})_{\mathcal{T}_h}&+ \|(h^{-1}+\tau_1-\frac 12 \bm{\beta} \cdot\bm n)^{\frac 12 } (\varepsilon_h^y-\varepsilon_h^{\widehat{y}})\|_{\partial\mathcal T_h}^2 +\frac 1 2\| (-\nabla\cdot\bm{\beta})^{\frac  1 2}\varepsilon_h^y\|_{\mathcal T_h}^2\\
	&= -(\nabla\cdot \delta^{\bm q},\varepsilon_h^y)_{\mathcal{T}_h}-\langle \bm \Pi  \bm q\cdot\bm n,\varepsilon_h^{\widehat{y}} \rangle_{\partial\mathcal{T}_h}\\
	&\quad + ( \bm \beta \delta^y, \nabla \varepsilon_h^y)_{{\mathcal{T}_h}}
	+(\nabla\cdot\bm{\beta}\delta^y,\varepsilon_h^y)_{\mathcal T_h} -\langle \widehat {\bm\delta}_1,\varepsilon_h^y - \varepsilon_h^{\widehat y}\rangle_{\partial\mathcal T_h} \\
	&=: T_1 + T_2 + T_3 +T_4.
	\end{split}
	\end{equation}
	
	We rewrite the term $T_1$ using the interpolation operator $ \mathcal I_h $:
	\begin{align*}
	T_1& = -(\nabla\cdot \delta^{\bm q},\varepsilon_h^y)_{\mathcal{T}_h}-\langle \bm \Pi  \bm q\cdot\bm n,\varepsilon_h^{\widehat{y}} \rangle_{\partial\mathcal{T}_h}\\
	& =-(\nabla\cdot {\bm q},\varepsilon_h^y)_{\mathcal{T}_h} + (\nabla\cdot {\bm \Pi \bm q},\varepsilon_h^y)_{\mathcal{T}_h}-\langle \bm \Pi  \bm q\cdot\bm n,\varepsilon_h^{\widehat{y}} \rangle_{\partial\mathcal{T}_h}\\
	& = -(\nabla\cdot {\bm q},\varepsilon_h^y -  \mathcal I_h(\varepsilon_h^y, \varepsilon_h^{\widehat{y}}))_{\mathcal{T}_h} -(\nabla\cdot {\bm q}, \mathcal I_h(\varepsilon_h^y, \varepsilon_h^{\widehat{y}}))_{\mathcal{T}_h}\\
	&\quad+(\nabla \cdot \bm \Pi \bm q,\varepsilon_h^y)-\langle \bm \Pi  \bm q\cdot\bm n,\varepsilon_h^{\widehat{y}} \rangle_{\partial\mathcal{T}_h}\\
	& = -(\nabla\cdot\bm q,\varepsilon_h^y -  \mathcal I_h (\varepsilon_h^y, \varepsilon_h^{\widehat{y}}))_{\mathcal{T}_h} + ( \bm q,  \nabla \mathcal I_h(\varepsilon_h^y, \varepsilon_h^{\widehat{y}}))_{\mathcal{T}_h} \\
	&\quad+(\nabla \cdot \bm \Pi \bm q,\varepsilon_h^y)-\langle \bm \Pi  \bm q\cdot\bm n,\varepsilon_h^{\widehat{y}} \rangle_{\partial\mathcal{T}_h}\\
	& = -(\nabla\cdot\bm q,\varepsilon_h^y -  \mathcal I_h (\varepsilon_h^y, \varepsilon_h^{\widehat{y}}))_{\mathcal{T}_h} + ( \delta^{\bm q},  \nabla \mathcal I_h(\varepsilon_h^y, \varepsilon_h^{\widehat{y}}))_{\mathcal{T}_h} \\
	&\quad+( \bm \Pi {\bm q},  \nabla \mathcal I_h(\varepsilon_h^y, \varepsilon_h^{\widehat{y}}))_{\mathcal{T}_h}  + (\nabla \cdot \bm \Pi \bm q,\varepsilon_h^y)-\langle \bm \Pi  \bm q\cdot\bm n,\varepsilon_h^{\widehat{y}} \rangle_{\partial\mathcal{T}_h}\\
	& = -(\nabla\cdot\bm q,\varepsilon_h^y -  \mathcal I_h (\varepsilon_h^y, \varepsilon_h^{\widehat{y}}))_{\mathcal{T}_h} + ( \delta^{\bm q},  \nabla \mathcal I_h(\varepsilon_h^y, \varepsilon_h^{\widehat{y}}))_{\mathcal{T}_h}.
	\end{align*}
	The last step holds since
	\begin{align*}
	( \bm \Pi {\bm q},  \nabla \mathcal I_h(\varepsilon_h^y, \varepsilon_h^{\widehat{y}}))_{\mathcal{T}_h}  &= \langle \bm \Pi  \bm q\cdot\bm n,\mathcal I_h(\varepsilon_h^y, \varepsilon_h^{\widehat{y}}) \rangle_{\partial\mathcal{T}_h} - 	( \nabla \cdot \bm \Pi {\bm q}, \mathcal I_h(\varepsilon_h^y, \varepsilon_h^{\widehat{y}}))_{\mathcal{T}_h} \\
	&= \langle \bm \Pi  \bm q\cdot\bm n, \varepsilon_h^{\widehat{y}} \rangle_{\partial\mathcal{T}_h} - 	( \nabla \cdot \bm \Pi {\bm q}, \varepsilon_h^y)_{\mathcal{T}_h}.
	\end{align*}
	This implies
	\begin{align*}
	T_1&\le \|\nabla \cdot \bm q\|_{\mathcal T_h} \|\varepsilon_h^y -  \mathcal I_h (\varepsilon_h^y, \varepsilon_h^{\widehat{y}})\|_{\mathcal T_h} + \|\delta^{\bm q}\|_{\mathcal T_h} \|\nabla\mathcal I_h (\varepsilon_h^y, \varepsilon_h^{\widehat{y}})\|_{\mathcal T_h}\\
	&\lesssim  h (\|\varepsilon_h^{\bm q}\|_{\mathcal T_h} + h^{-\frac 1 2}\| \varepsilon_h^y - \varepsilon_h^{\widehat y}\|_{\partial \mathcal T_h}) + h^{s_{\bm q}}  \|\bm q\|_{s^{\bm q},\Omega} (\|\varepsilon_h^{\bm q}\|_{\mathcal T_h} + h^{-\frac 1 2}\| \varepsilon_h^y - \varepsilon_h^{\widehat y}\|_{\partial \mathcal T_h})\\
	&\lesssim h^{s_{\bm q}}  \|\bm q\|_{s^{\bm q},\Omega} (\|\varepsilon_h^{\bm q}\|_{\mathcal T_h} + h^{-\frac 1 2}\| \varepsilon_h^y - \varepsilon_h^{\widehat y}\|_{\partial \mathcal T_h}).
	\end{align*}
	Note that we used $ s_{\bm q} \in [0,1] $.
	
	For the terms $T_2$, $T_3$, and $T_4 $, apply \eqref{nabla_y} and Young's inequality to obtain
	\begin{align*}
	T_2 &=  ( \bm \beta \delta^y, \nabla \varepsilon_h^y)_{{\mathcal{T}_h}} \le C \| \delta^y\|_{\mathcal T_h}^2 + \frac 1 4
	\|\varepsilon_h^{\bm{q}}\|_{\mathcal T_h}^2 + \frac 1 {4h} \|{\varepsilon_h^y-\varepsilon_h^{\widehat{y}}}\|_{\partial \mathcal T_h}^2,\\
	T_3 &= (\nabla\cdot\bm{\beta}\delta^y,\varepsilon_h^y)_{\mathcal T_h} \le C \|\delta^y\|_{\mathcal T_h}^2 + \frac 1 2\|(-\nabla\cdot\bm{\beta})^{\frac 1 2} \varepsilon_h^y\|_{\mathcal T_h}^2,\\
	T_4 &= - \langle \widehat {\bm\delta}_1,\varepsilon_h^y - \varepsilon_h^{\widehat y}\rangle_{\partial\mathcal T_h} \le 4h \|\widehat {\bm\delta}_1\|_{\partial\mathcal T_h}^2 + \frac 1 {4h} \|{\varepsilon_h^y-\varepsilon_h^{\widehat{y}}}\|_{\partial \mathcal T_h}^2.
	\end{align*}
	Summing the estimates for $\{T_i\}_{i=1}^4$ gives the result.
\end{proof}
\begin{remark}
	In Part I \cite{HuMateosSinglerZhangZhang1}, we defined $\widehat {\bm\delta}_1 = \delta^{\bm q} \cdot \bm n +  \bm{\beta}\cdot\bm n \delta^{\widehat y} + (h^{-1}+\tau_1)(\delta^y - \delta^{\widehat y})$.  It is not meaningful to estimate $ \|\widehat {\bm\delta}_1\|_{\partial\mathcal T_h}$ if we only assume $r_{\bm q}\geq 0$.  In this paper, we have $ \widehat {\bm\delta}_1 = \bm{\beta}\cdot\bm n \delta^{\widehat y} + (h^{-1}+\tau_1)(\delta^y - \delta^{\widehat y}) $, and we can estimate $ \|\widehat {\bm\delta}_1\|_{\partial\mathcal T_h}$.
\end{remark}

\subsubsection{Step 3: Estimate for $\varepsilon_h^{y}$ by a duality argument}
\label{subsec:proof_step3}

Next, for any $\Theta$ in $L^2(\Omega)$ we consider the dual problem 
\begin{equation}\label{Dual_PDE}
\begin{split}
\bm\Phi-\nabla\Psi &= 0\qquad \ \ ~\text{in}\ \  \Omega,\\
\nabla\cdot\bm{\Phi}+\nabla\cdot(\bm\beta\Psi) &= \Theta \qquad \ \text{in}\  \ \Omega,\\
\Psi &= 0\qquad \ \ ~\text{on}\ \partial\Omega.
\end{split}
\end{equation}
Since the domain $\Omega$ is convex, we have the regularity estimate
\begin{align}
\norm{\bm \Phi}_{1,\Omega} + \norm{\Psi}_{2,\Omega} \le C_{\text{reg}} \norm{\Theta}_\Omega.
\end{align}

We use the following notation in the next proof for the estimate of $\varepsilon_h^y$:
\begin{align}
\delta^{\bm \Phi} &=\bm \Phi-{\bm\Pi} \bm \Phi, \quad \delta^\Psi=\Psi- {\Pi} \Psi, \quad
\delta^{\widehat \Psi} = \Psi-P_M\Psi.
\end{align}
\begin{lemma}\label{e_sec}
	We have
	\begin{align*}
	\|\varepsilon_h^y\|_{\mathcal T_h} \lesssim  h^{s_{\bm q}+1}\norm{\bm q}_{s^{\bm q},\Omega} + h^{s_{y}}\norm{y}_{s^{y},\Omega}.
	\end{align*}
\end{lemma}

\begin{proof}
	We take $\Theta =- \varepsilon_h^y$ in the dual problem \eqref{Dual_PDE} and $(\bm r_1,w_1,\mu_1) = ( {\bm\Pi}\bm{\Phi},{\Pi}\Psi,\\P_M\Psi)$ in the error equation \eqref{error_equation_L2k1} in \Cref{lemma:step1_first_lemma}.  Since $\Psi = 0$ on $\varepsilon_h^{\partial}$, we have
	\begin{align*}
	%%%%%%%%%%
	\hspace{1em}&\hspace{-1em}  \mathscr B_1 (\varepsilon^{\bm q}_h,\varepsilon^y_h,\varepsilon^{\widehat y}_h;{\bm\Pi}\bm{\Phi},{\Pi}\Psi,P_M\Psi)\\
	%%%%%%%%%%%%%%%%%%%
	&= (\varepsilon^{\bm q}_h,{\bm\Pi}\bm{\Phi})_{\mathcal T_h}-( \varepsilon^y_h,\nabla\cdot{\bm\Pi}\bm{\Phi})_{\mathcal T_h}+\langle  \varepsilon^{\widehat y}_h,{\bm\Pi}\bm{\Phi}\cdot\bm n\rangle_{\partial\mathcal T_h\backslash \varepsilon_h^\partial}\\
	%%%%%%%%%%%%%%%
	&\quad +(\nabla\cdot\varepsilon_h^{\bm q},\Pi \Psi)_{\mathcal{T}_h} -(\bm \beta\varepsilon^y_h,  \nabla {\Pi}\Psi)_{\mathcal T_h}-(\nabla\cdot\bm\beta \varepsilon^y_h,{\Pi}\Psi)_{\mathcal T_h}+\langle (h^{-1}+\tau_1) \varepsilon^y_h ,{\Pi}\Psi\rangle_{\partial\mathcal T_h}\\
	%%%%%%%%%%%%%%%%%%
	&\quad+\langle (\bm\beta\cdot\bm n -h^{-1}-\tau_1) \varepsilon^{\widehat y}_h,{\Pi}\Psi\rangle_{\partial\mathcal T_h}\\
	%%%%%%%%%%%%%%%%%%%%
	&\quad-\langle  \varepsilon^{\bm q}_h\cdot\bm n+\bm \beta\cdot\bm n\varepsilon^{\widehat y}_h  + (h^{-1}+\tau_1)(\varepsilon^y_h - \varepsilon^{\widehat y}_h),P_M\Psi\rangle_{\partial\mathcal T_h}\\
	%%%%%%%%%%%%%%%%%%%%%%%
	&= (\varepsilon^{\bm q}_h,\bm{\Phi})_{\mathcal {T}_h}-( \varepsilon^y_h,\nabla\cdot \bm{\Phi})_{\mathcal T_h} + ( \varepsilon^y_h,\nabla\cdot\delta^{\bm{\Phi}})_{\mathcal T_h}+\langle  \varepsilon^{\widehat y}_h,\bm \Pi \bm \Phi \cdot\bm n\rangle_{\partial\mathcal T_h}-(\varepsilon^{\bm q}_h , \nabla \Psi)_{\mathcal T_h}\\
	&\quad+\langle \varepsilon_h^{\bm q}\cdot \bm n, \Psi \rangle_{\partial \mathcal{T}_h}-(\bm \beta \varepsilon_h^y,\nabla \Psi)_{\mathcal{T}_h}+(\bm \beta \varepsilon_h^y,\nabla \delta^{\Psi})_{\mathcal{T}_h}-(\nabla \cdot \bm \beta \varepsilon_h^y,\Psi)_{\mathcal{T}_h}\\
	& \quad +(\nabla \cdot \bm \beta \varepsilon_h^y,\delta^{\Psi})_{\mathcal{T}_h}-\langle \varepsilon_h^{\bm q}\cdot \bm n,P_M \Psi \rangle_{\partial \mathcal{T}_h}-\langle \bm \beta\cdot\bm n\varepsilon^{\widehat y}_h ,\delta^{\Psi} \rangle_{\partial \mathcal{T}_h} \\
	&\quad- \langle  (h^{-1}+ \tau_1)(\varepsilon^y_h - \varepsilon^{\widehat y}_h),\delta^\Psi - \delta^{\widehat\Psi}\rangle_{\partial\mathcal T_h}.
	\end{align*}
	Here we used $\langle \bm \beta\cdot \bm n \varepsilon^{\widehat y}_h, \Psi\rangle_{\partial\mathcal T_h}=0$ and $\langle \bm \beta\cdot \bm n \varepsilon^{\widehat y}_h, P_M\Psi\rangle_{\partial\mathcal T_h}=0$, which both hold since $\varepsilon^{\widehat y}_h$ is a single-valued function on interior edges and $\varepsilon^{\widehat y}_h=0$ on $\varepsilon^{\partial}_h$.
	
	By the same argument as in the proof of	\Cref{lemma:step2_main_lemma} for the term $T_1$, we have 
	\begin{align*}
	( \varepsilon^y_h,\nabla\cdot\delta^{\bm{\Phi}})_{\mathcal T_h}&+\langle  \varepsilon^{\widehat y}_h,\bm \Pi \bm \Phi\cdot\bm n\rangle_{\partial\mathcal T_h}\\
	&=(\varepsilon_h^y-\mathcal I_h(\varepsilon_h^y,\varepsilon_h^{\widehat{y}}),\nabla\cdot  \bm\Phi)_{\mathcal{T}_h}-(\nabla\mathcal I_h(\varepsilon_h^y,\varepsilon_h^{\widehat{y}}),\delta^{\bm\Phi})_{\mathcal{T}_h}.
	\end{align*}
	Next, integration by parts gives
	\begin{align*}
	(\bm\beta \varepsilon_h^y, \nabla \delta^{ \Psi})_{\mathcal{T}_h}&=\langle \bm{\beta} \cdot\bm n \varepsilon_h^y, \delta^{ \Psi}\rangle_{\partial\mathcal T_h}- (\nabla\cdot \bm{\beta} \varepsilon_h^y, \delta^{ \Psi})_{\mathcal T_h} - (\bm{\beta}\cdot \nabla\varepsilon_h^y, \delta^{ \Psi})_{\mathcal T_h}.
	\end{align*}
	This implies
	\begin{align*}
	\hspace{1em}&\hspace{-1em}  \mathscr B_1 (\varepsilon^{\bm q}_h,\varepsilon^y_h,\varepsilon^{\widehat y}_h;{\bm\Pi}\bm{\Phi},{\Pi}\Psi,P_M\Psi)\\
	&=\|  \varepsilon_h^y\|_{\mathcal T_h}^2 + \langle \bm \beta\cdot\bm n(\varepsilon^y_h - \varepsilon^{\widehat y}_h),\delta^{\Psi} \rangle_{\partial\mathcal T_h} - (\nabla \varepsilon_h^y,\bm{\beta}\delta^{\Psi})_{\mathcal T_h}\\
	&\quad+(\varepsilon_h^y-\mathcal I_h(\varepsilon_h^y,\varepsilon_h^{\widehat{y}}),\nabla\cdot\delta^{\bm \Phi})_{\mathcal{T}_h}-(\nabla\mathcal I_h(\varepsilon_h^y,\varepsilon_h^{\widehat{y}}),\delta^{\bm\Phi})_{\mathcal{T}_h}\\
	& \quad- \langle   h^{-1}(\varepsilon^y_h-\varepsilon^{\widehat y}_h) + \tau_1(\varepsilon^y_h-\varepsilon^{\widehat y}_h),\delta^\Psi - \delta^{\widehat\Psi}\rangle_{\partial\mathcal T_h}.
	\end{align*}
	Also, since $\Psi = 0$ on $\varepsilon_h^\partial $, the error equation \eqref{error_equation_L2k1} in \Cref{lemma:step1_first_lemma} gives
	\begin{align*}
	\hspace{2em}&\hspace{-2em}  \mathscr B_1 (\varepsilon^{\bm q}_h,\varepsilon^y_h,\varepsilon^{\widehat y}_h;{\bm\Pi}\bm{\Phi},{\Pi}\Psi,P_M\Psi)\\
	&=-(\nabla\cdot \delta^{\bm q},{\Pi}\Psi)_{\mathcal{T}_h}-\langle \bm \Pi \bm q\cdot\bm n,P_M \Psi \rangle_{\mathcal{T}_h}\nonumber\\
	& \quad  +( \bm \beta \delta^y, \nabla {\Pi}\Psi)_{{\mathcal{T}_h}}
	+(\nabla\cdot\bm{\beta}\delta^y,{\Pi}\Psi)_{\mathcal T_h} - \langle \widehat{\bm \delta}_1, {\Pi}\Psi-P_M \Psi \rangle_{\partial{{\mathcal{T}_h}}}\\
	&=-(\nabla\cdot {\bm q},\Pi {\Psi})_{\mathcal T_h} +(\nabla\cdot \bm \Pi \bm q, {\Psi})_{\mathcal T_h} -\langle \bm \Pi \bm q\cdot\bm n, \Psi \rangle_{\mathcal{T}_h} \\
	&\quad+( \bm \beta \delta^y, \nabla {\Pi}\Psi)_{{\mathcal{T}_h}}+(\nabla\cdot\bm{\beta}\delta^y,{\Pi}\Psi)_{\mathcal T_h}- \langle \widehat{\bm \delta}_1, {\Pi}\Psi-P_M \Psi \rangle_{\partial{{\mathcal{T}_h}}},\\
	&=(\nabla\cdot {\bm q},\delta^{\Psi})_{\mathcal T_h} - (\nabla\cdot {\bm q},\Psi)_{\mathcal T_h} + (\nabla\cdot {\bm \Pi \bm q},\Psi)_{\mathcal T_h} -\langle \bm \Pi \bm q\cdot\bm n,\Psi \rangle_{\mathcal{T}_h} \\
	&\quad+( \bm \beta \delta^y, \nabla {\Pi}\Psi)_{{\mathcal{T}_h}}+(\nabla\cdot\bm{\beta}\delta^y,{\Pi}\Psi)_{\mathcal T_h}- \langle \widehat{\bm \delta}_1, {\Pi}\Psi-P_M \Psi \rangle_{\partial{{\mathcal{T}_h}}},\\
	&=(\nabla\cdot\bm q,\delta^{\Psi})_{\mathcal T_h} + ({\bm q},\nabla\Psi)_{\mathcal T_h} - ({\bm \Pi \bm q},\nabla\Psi)_{\mathcal T_h} \\
	&\quad+( \bm \beta \delta^y, \nabla {\Pi}\Psi)_{{\mathcal{T}_h}}+(\nabla\cdot\bm{\beta}\delta^y,{\Pi}\Psi)_{\mathcal T_h}- \langle \widehat{\bm \delta}_1, {\Pi}\Psi-P_M \Psi \rangle_{\partial{{\mathcal{T}_h}}},\\
	&=(\nabla \cdot \bm q,\delta^{\Psi})+(\delta^{\bm q},\nabla \delta^\Psi)_{\mathcal{T}_h}+( \bm \beta \delta^y, \nabla {\Pi}\Psi)_{{\mathcal{T}_h}}\\
	&\quad+(\nabla\cdot\bm{\beta}\delta^y,{\Pi}\Psi)_{\mathcal T_h}- \langle \widehat{\bm \delta}_1, {\Pi}\Psi-P_M \Psi \rangle_{\partial{{\mathcal{T}_h}}}.
	\end{align*}
	The two equalities above give
	\begin{align*}
	\|  \varepsilon_h^y\|_{\mathcal T_h}^2  &= - \langle \bm{\beta}\cdot\bm n (\varepsilon^y_h - \varepsilon^{\widehat y}_h), \delta^{\Psi} \rangle_{\partial\mathcal T_h} + (\nabla \varepsilon_h^y,\bm{\beta}\delta^{\Psi})_{{\mathcal{T}_h}}+( \bm \beta \delta^y, \nabla {\Pi}\Psi)_{{\mathcal{T}_h}}\\
	& \quad +(\nabla\cdot\bm{\beta}\delta^y,{\Pi}\Psi)_{\mathcal T_h}
	+\langle   (h^{-1}+\tau_1)(\varepsilon^y_h-\varepsilon^{\widehat y}_h) + \widehat{\bm \delta}_1,\delta^\Psi - \delta^{\widehat\Psi}\rangle_{\partial\mathcal T_h}\\
	&\quad -(\varepsilon_h^y-\mathcal I_h(\varepsilon_h^y,\varepsilon_h^{\widehat{y}}),\nabla\cdot\delta^{\bm \Phi})_{\mathcal{T}_h}+(\nabla\mathcal I_h(\varepsilon_h^y,\varepsilon_h^{\widehat{y}}),\delta^{\bm\Phi})_{\mathcal{T}_h}\\
	&\quad +(\nabla \cdot\bm q,\delta^{\Psi})+(\delta^{\bm q},\nabla \delta^\Psi)_{\mathcal{T}_h}\\
	&=:\sum_{i=1}^9 R_i.
	\end{align*}
	Bounds for $R_1$ to $R_5$ have been obtained in Part I \cite{HuMateosSinglerZhangZhang1}; we have 
	\begin{align*}
	\sum_{i=1}^5 R_i  \lesssim  ( h^{s_{\bm q}+1}\norm{\bm q}_{s^{\bm q},\Omega} + h^{s_{y}}\norm{y}_{s^{y},\Omega}) \|\varepsilon_h^y\|_{\mathcal T_h}.
	\end{align*}
	For the terms $R_6$ and $R_7$, Lemma \ref{error_esti_inter} and Lemma \ref{lemma:step2_main_lemma} give
	\begin{align*}
	R_6&= -(\varepsilon_h^y-\mathcal I_h(\varepsilon_h^y,\varepsilon_h^{\widehat{y}}),\nabla\cdot \bm \Phi)_{\mathcal{T}_h}\\
	&\le \| \varepsilon_h^y-\mathcal I_h(\varepsilon_h^y,\varepsilon_h^{\widehat{y}})\|_{\mathcal{T}_h} \|\nabla\cdot \bm \Phi\|_{\mathcal{T}_h}\\
	&\lesssim h (\norm{\varepsilon_h^{\bm{q}}}_{\mathcal{T}_h}+h^{-\frac 12}\|{\varepsilon_h^y-\varepsilon_h^{\widehat{y}}}\|_{\partial \mathcal T_h}) \|\nabla\cdot \bm \Phi\|_{\mathcal{T}_h}\\
	& \lesssim  ( h^{s_{\bm q}+1}\norm{\bm q}_{s^{\bm q},\Omega} + h^{s_{y}}\norm{y}_{s^{y},\Omega}) \|\varepsilon_h^y\|_{\mathcal T_h},\\
	R_7& = (\nabla\mathcal I_h(\varepsilon_h^y,\varepsilon_h^{\widehat{y}}),\delta^{\bm\Phi})_{\mathcal{T}_h}\\
	& \le \|\nabla\mathcal I_h(\varepsilon_h^y,\varepsilon_h^{\widehat{y}})\|_{\mathcal{T}_h} \|\delta^{\bm\Phi}\|_{\mathcal{T}_h}\\
	&\lesssim h (\norm{\varepsilon_h^{\bm{q}}}_{\mathcal{T}_h}+h^{-\frac 12}\|{\varepsilon_h^y-\varepsilon_h^{\widehat{y}}}\|_{\partial \mathcal T_h}) \|\delta^{\bm \Phi}\|_{\mathcal{T}_h}\\
	& \lesssim  ( h^{s_{\bm q}+1}\norm{\bm q}_{s^{\bm q},\Omega} + h^{s_{y}}\norm{y}_{s^{y},\Omega}) \|\varepsilon_h^y\|_{\mathcal T_h}.
	\end{align*}
	For $R_8$, we have
	\begin{align*}
	R_8 &\le \| \nabla\cdot\bm q \|_{\mathcal{T}_h}\| \delta^{\Psi} \|_{\mathcal{T}_h} \lesssim h^2\| \Psi \|_{2,\Omega}\\
	&\lesssim h^2\| \varepsilon_h^y \|_{\mathcal{T}_h}.
	\end{align*}
	Applying the triangle inequality for $R_9$ gives
	\begin{align*}
	R_9 &\le  \| \delta^{\bm q} \|_{\mathcal T_h} \| \nabla \delta ^{\Psi} \|_{\mathcal T_h} \lesssim   h^{s_{\bm q}+1}\norm{\bm q}_{s^{\bm q},\Omega}  \|\varepsilon_h^y\|_{\mathcal T_h}.
	\end{align*}
	Using $ s_{\bm q} \in [0,1] $ and summing the estimates for $R_1$ to $R_9$ completes the proof.
\end{proof}

The triangle inequality gives optimal convergence rates for $\|\bm q -\bm q_h(u)\|_{\mathcal T_h}$ and $\|y -y_h(u)\|_{\mathcal T_h}$:
\begin{lemma}\label{lemma:step3_conv_rates}
	\begin{subequations}
		\begin{align}
		\|\bm q -\bm q_h(u)\|_{\mathcal T_h}&\le \|\delta^{\bm q}\|_{\mathcal T_h} + \|\varepsilon_h^{\bm q}\|_{\mathcal T_h} \lesssim  h^{s_{\bm q}}\norm{\bm q}_{s^{\bm q},\Omega} + h^{s_{y}-1}\norm{y}_{s^{y},\Omega},\\
		\|y -y_h(u)\|_{\mathcal T_h}&\le \|\delta^{y}\|_{\mathcal T_h} + \|\varepsilon_h^{y}\|_{\mathcal T_h} \lesssim h^{s_{\bm q}+1}\norm{\bm q}_{s^{\bm q},\Omega} + h^{s_{y}}\norm{y}_{s^{y},\Omega}.
		\end{align}
	\end{subequations}
\end{lemma}

\subsubsection{Step 4: The error equation for part 2 of the auxiliary problem \eqref{HDG_u_b}}

Next, we estimate the error between the exact state $z$ and flux $\bm p$ satisfying the mixed form \eqref{mixed_a}-\eqref{mixed_d} of the optimality system and the solutions $ z_h(u) $ and $ \bm p_h(u) $ of the auxiliary problem. Define
%
%
%Next, we focus on the dual variables, i.e., the state $z$ and the flux $\bm p$, and estimate the error between the solutions of the auxiliary problem and the mixed form \eqref{mixed_a} -  \eqref{mixed_d} of the optimality system. Define
%
\begin{align}\label{notation_3}
\begin{aligned}%column 1
\delta^{\bm p} &=\bm p-{\bm\Pi}\bm p,\\
\delta^z&=z- \Pi z,\\
\delta^{\widehat z} &= z-P_M z,\\
\widehat {\bm\delta}_2 &= -\bm{\beta}\cdot\bm n \delta^{\widehat z} + (h^{-1}+\tau_2)(\delta^z - \delta^{\widehat z}),
\end{aligned}
&&
\begin{aligned}%column 2
\varepsilon^{\bm p}_h &= {\bm\Pi} \bm p-\bm p_h(u),\\
\varepsilon^{z}_h &= \Pi z-z_h(u),\\
\varepsilon^{\widehat z}_h &= P_M z-\widehat{z}_h(u),\\
\ & \
\end{aligned}
\end{align}
where $\widehat z_h(u) = \widehat z_h^o(u)$ on $\varepsilon_h^o$ and $\widehat z_h(u) = 0$ on $\varepsilon_h^{\partial}$.  This gives $\varepsilon_h^{\widehat z} = 0$ on $\varepsilon_h^{\partial}$. 

\begin{lemma}\label{lemma:step4_first_lemma}
	We have
	\begin{align}
	\hspace{3em}&\hspace{-3em}  \mathscr B_2 (\varepsilon_h^{\bm p},\varepsilon_h^{z}, \varepsilon_h^{\widehat z}, \bm r_2, w_2, \mu_2) \nonumber\\
	&=-(\nabla\cdot\delta^{\bm p},w_2)_{\mathcal{T}_h}-\langle \bm \Pi \bm p\cdot\bm n,\mu_2 \rangle_{\partial \mathcal{T}_h\backslash\varepsilon_h^{\partial}}-( \bm \beta \delta^z, \nabla w_2)_{{\mathcal{T}_h}} \nonumber\\
	&\quad+(y-y_h(u),w_2)_{\mathcal T_h}- \langle \widehat{\bm \delta}_2, w_2 \rangle_{\partial{{\mathcal{T}_h}}} + \langle \widehat{\bm \delta}_2, \mu_2 \rangle_{\partial{{\mathcal{T}_h}}\backslash \varepsilon_h^{\partial}}.\label{error_equation_L2k2}
	\end{align}
\end{lemma}
The proof is similar to the proof of \Cref{lemma:step4_first_lemma} and is omitted.

\subsubsection{Step 5: Estimate for $\varepsilon_h^{\boldmath p}$}
We use the following discrete Poincar{\'e} inequality from \cite{HuMateosSinglerZhangZhang1} to estimate $\varepsilon_h^{\bm p}$.
\begin{lemma}\label{lemma:discr_Poincare_ineq}
	We have
	\begin{align}\label{poin_in}
	\|\varepsilon_h^z\|_{\mathcal T_h} \le C(\|\nabla \varepsilon_h^z\|_{\mathcal T_h} + h^{-\frac 1 2} \|\varepsilon_h^z - \varepsilon_h^{\widehat z}\|_{\partial\mathcal T_h}).
	\end{align}
\end{lemma}

\begin{lemma}\label{lemma:step5_main_lemma}
	We have
	\begin{subequations}
		\begin{align}
		\hspace{3em}&\hspace{-3em}\norm{\varepsilon_h^{\bm{p}}}_{\mathcal{T}_h}+h^{-\frac 1 2}\|{\varepsilon_h^z-\varepsilon_h^{\widehat{z}}}\|_{\partial \mathcal T_h}\nonumber \\ 
		&\lesssim  h^{s_{\bm p}}\norm{\bm p}_{s^{\bm p},\Omega} + h^{s_{z}-1}\norm{z}_{s^{z},\Omega} +h^{s_{\bm q}+1}\norm{\bm q}_{s^{\bm q},\Omega} + h^{s_{y}}\norm{y}_{s^{y},\Omega},\label{error_p}\\
		\norm{\varepsilon_h^{{z}}}_{\mathcal{T}_h} &\lesssim  h^{s_{\bm p}}\norm{\bm p}_{s^{\bm p},\Omega} + h^{s_{z}-1}\norm{z}_{s^{z},\Omega} +h^{s_{\bm q}+1}\norm{\bm q}_{s^{\bm q},\Omega} + h^{s_{y}}\norm{y}_{s^{y},\Omega}.\label{error_z}
		\end{align}
	\end{subequations}
\end{lemma}

\begin{proof}
	First, take $ (\bm v_h,w_h,\mu_h) = (\varepsilon_h^{\bm p},\varepsilon_h^{z}, \varepsilon_h^{\widehat z}) $ in the key inequality in \Cref{nabla_ine} to get
	\begin{align}\label{nabla_z}
	\|\nabla \varepsilon_h^z\|_{\mathcal T_h}  \lesssim  \|\varepsilon^{\bm p}_h\|_{\mathcal T_h}+h^{-\frac1 2}\|\varepsilon^z_h-\varepsilon^{\widehat z}_h\|_{\partial\mathcal T_h}.
	\end{align}
	Next, since $\varepsilon_h^{\widehat z}=0$ on $\varepsilon_h^\partial$, the energy identity for $ \mathscr B_2 $ in \Cref{property_B} gives
	\begin{align*}
	\hspace{1em}&\hspace{-1em} \mathscr B_2 (\varepsilon_h^{\bm p},\varepsilon_h^{ z}, \varepsilon_h^{\widehat z}, \varepsilon_h^{\bm p},\varepsilon_h^{z}, \varepsilon_h^{\widehat z}) \\
	&= (\varepsilon_h^{\bm{p}},\varepsilon_h^{\bm{p}})_{\mathcal{T}_h}+ \|(h^{-1}+\tau_2+\frac 12 \bm{\beta} \cdot\bm n)^{\frac 12 } (\varepsilon_h^z-\varepsilon_h^{\widehat{z}})\|_{\partial\mathcal T_h}^2+\frac 1 2\| (-\nabla\cdot\bm{\beta})^{\frac  1 2}\varepsilon_h^z\|_{\mathcal T_h}^2. 
	\end{align*}
	Take $(\bm r_2, w_2,\mu_2) = (\bm \varepsilon_h^{\bm p},\varepsilon_h^z,\varepsilon_h^{\widehat z})$ in the error equation \eqref{error_equation_L2k2} in \Cref{lemma:step4_first_lemma} to obtain
	\begin{align*}
	\hspace{4em}&\hspace{-4em} (\varepsilon_h^{\bm{p}},\varepsilon_h^{\bm{p}})_{\mathcal{T}_h} + \|(h^{-1}+\tau_2+\frac 12 \bm{\beta} \cdot\bm n)^{\frac 12 } (\varepsilon_h^z-\varepsilon_h^{\widehat{z}})\|_{\partial\mathcal T_h}^2+\frac 1 2\| (-\nabla\cdot\bm{\beta})^{\frac  1 2}\varepsilon_h^z\|_{\mathcal T_h}^2\\
	&= -(\nabla\cdot\delta^{\bm p},\varepsilon_h^z)_{\mathcal{T}_h}-\langle \bm \Pi \bm p\cdot\bm n,\varepsilon_h^{\widehat{z}} \rangle_{\partial \mathcal{T}_h}\\
	&\quad-( \bm \beta \delta^z, \nabla \varepsilon_h^z)_{{\mathcal{T}_h}}  -\langle \widehat {\bm\delta}_2,\varepsilon_h^z - \varepsilon_h^{\widehat z}\rangle_{\partial\mathcal T_h} + (y-y_h(u),\varepsilon_h^z)_{\mathcal T_h}\\
	& =: T_1+T_2+T_3+T_4.
	\end{align*}
	By the same argument as in the proof of  \Cref{lemma:step2_main_lemma}, apply \eqref{nabla_z} and Young's inequality to obtain
	\begin{align*}
	T_1 &=-(\nabla\cdot\delta^{\bm p},\varepsilon_h^z)_{\mathcal{T}_h}-\langle \bm \Pi \bm p\cdot\bm n,\varepsilon_h^{\widehat{z}} \rangle_{\partial \mathcal{T}_h}\\
	&=-(\nabla\cdot\bm p,\varepsilon_h^z-\mathcal I_h(\varepsilon_h^z,\varepsilon_h^{\widehat{z}}))_{\mathcal{T}_h}+(\delta^{\bm p},\nabla\mathcal I_h(\varepsilon_h^z,\varepsilon_h^{\widehat{z}}))_{\mathcal{T}_h}\\
	&=-(\nabla\cdot\delta^{\bm p},\varepsilon_h^z-\mathcal I_h(\varepsilon_h^z,\varepsilon_h^{\widehat{z}}))_{\mathcal{T}_h}+(\delta^{\bm p},\nabla\mathcal I_h(\varepsilon_h^z,\varepsilon_h^{\widehat{z}}))_{\mathcal{T}_h}\\
	&\le h\| \nabla\cdot \delta^{\bm p}\|_{\mathcal T_h} h^{-1}\|  \varepsilon_h^z-\mathcal I_h(\varepsilon_h^z,\varepsilon_h^{\widehat{z}})\|_{\mathcal{T}_h}+\|\delta^{\bm p}\|_{\mathcal{T}_h} \|\nabla\mathcal I_h(\varepsilon_h^z,\varepsilon_h^{\widehat{z}})\|_{\mathcal{T}_h}\\
	&\le C h^2\| \nabla\cdot\delta^{\bm p} \|_{\mathcal{T}_h}^2+C\|\delta^{\bm p}\|_{\mathcal{T}_h}^2+ \frac 1 8 \|\varepsilon_h^{\bm p}\|_{\mathcal{T}_h}^2+\frac 1 {8h}\|\varepsilon_h^z-\varepsilon_h^{\widehat{z}}\|_{\partial \mathcal{T}_h}^2,\\
	T_2 &=  -( \bm \beta \delta^z, \nabla \varepsilon_h^z)_{{\mathcal{T}_h}} \le C \| \delta^z\|_{\mathcal T_h}^2 + \frac 1 8
	\|\varepsilon_h^{\bm{p}}\|_{\mathcal T_h}^2 + \frac 1 {8h} \|{\varepsilon_h^z-\varepsilon_h^{\widehat{z}}}\|_{\partial \mathcal T_h}^2,\\
	T_3 &=  -\langle \widehat {\bm\delta}_2,\varepsilon_h^z - \varepsilon_h^{\widehat z}\rangle_{\partial\mathcal T_h} \le 8h \|\widehat {\bm\delta}_2\|_{\partial\mathcal T_h}^2 + \frac 1 8
	\|\varepsilon_h^{\bm{p}}\|_{\mathcal T_h}^2 + \frac 1 {8h} \|{\varepsilon_h^z-\varepsilon_h^{\widehat{z}}}\|_{\partial \mathcal T_h}^2.
	\end{align*}
	For the term $T_4$, we have
	\begin{align*}
	T_4&= (y-y_h(u),\varepsilon_h^z)_{\mathcal T_h} \le  \|y-y_h(u)\|_{\mathcal T_h} \|\varepsilon_h^z\|_{\mathcal T_h}\\
	&\le C\|y-y_h(u)\|_{\mathcal T_h} (\|\nabla \varepsilon_h^z\|_{\mathcal T_h} + h^{-\frac 1 2} \|\varepsilon_h^z - \varepsilon_h^{\widehat z}\|_{\partial\mathcal T_h})\\
	&\le  C\|y-y_h(u)\|_{\mathcal T_h} (\|\varepsilon^{\bm p}_h\|_{\mathcal T_h}+h^{-\frac1 2}\|\varepsilon^z_h-\varepsilon^{\widehat z}_h\|_{\partial\mathcal T_h})\\
	&\le  C\|y-y_h(u)\|_{\mathcal T_h}^2 + \frac 1 8
	\|\varepsilon_h^{\bm{p}}\|_{\mathcal T_h}^2 + \frac 1 {8h} \|{\varepsilon_h^z-\varepsilon_h^{\widehat{z}}}\|_{\partial \mathcal T_h}^2.
	\end{align*}
	Summing $T_1$ to $T_4$ gives \eqref{error_p}; then \eqref{poin_in}, \eqref{error_p}, and \eqref{nabla_z} together imply \eqref{error_z}.
\end{proof}

The triangle inequality gives optimal convergence rates for $\|\bm p -\bm p_h(u)\|_{\mathcal T_h}$ and $\|z -z_h(u)\|_{\mathcal T_h}$:
\begin{lemma}\label{lemma:step6_conv_rates}
	\begin{subequations}
		\begin{align}
		\|\bm p -\bm p_h(u)\|_{\mathcal T_h}
		&\lesssim h^{s_{\bm p}}\norm{\bm p}_{s^{\bm p},\Omega}  + h^{s_{z}-1}\norm{z}_{s^{z},\Omega} +h^{s_{\bm q}+1}\norm{\bm q}_{s^{\bm q},\Omega} + h^{s_{y}}\norm{y}_{s^{y},\Omega},\\
		\|z -z_h(u)\|_{\mathcal T_h} &\lesssim  h^{s_{\bm p}}\norm{\bm p}_{s^{\bm p},\Omega} + h^{s_{z}-1}\norm{z}_{s^{z},\Omega} + h^{s_{\bm q}+1}\norm{\bm q}_{s^{\bm q},\Omega} + h^{s_{y}}\norm{y}_{s^{y},\Omega}.
		\end{align}
	\end{subequations}
\end{lemma}

\subsubsection{Step 6: Estimates for $\|u-u_h\|_{\varepsilon_h^\partial}$ and $\norm {y-y_h}_{\mathcal T_h}$}

To obtain the main result, we estimate the error between the solution of the auxiliary problem and the HDG discretized optimality system \eqref{HDG_full_discrete}.  Define
\begin{equation*}
\begin{split}
\zeta_{\bm q} &=\bm q_h(u)-\bm q_h,\quad\zeta_{y} = y_h(u)-y_h,\quad\zeta_{\widehat y} = \widehat y_h(u)-\widehat y_h,\\
\zeta_{\bm p} &=\bm p_h(u)-\bm p_h,\quad\zeta_{z} = z_h(u)-z_h,\quad\zeta_{\widehat z} = \widehat z_h(u)-\widehat z_h,
\end{split}
\end{equation*}
where $ \widehat y_h = \widehat y_h^o $ on $ \varepsilon_h^o $, $ \widehat y_h = u_h $ on $ \varepsilon_h^\partial $, $ \widehat z_h = \widehat z_h^o $ on $ \varepsilon_h^o $, and $ \widehat z_h = 0 $ on $ \varepsilon_h^\partial $.  This gives $ \zeta_{\widehat z} = 0 $ on $ \varepsilon_h^\partial $.

Subtracting the two problems gives the error equations
\begin{subequations}\label{eq_yh}
	\begin{align}
	\mathscr B_1(\zeta_{\bm q},\zeta_y,\zeta_{\widehat y};\bm r_1, w_1,\mu_1)&=-\langle P_M u-u_h, \bm r_1\cdot \bm n + (\bm{\beta} \cdot\bm n -h^{-1}-\tau_1)w_1\rangle_{\varepsilon_h^\partial},\label{eq_yh_yhu}\\
	\mathscr B_2(\zeta_{\bm p},\zeta_z,\zeta_{\widehat z};\bm r_2, w_2,\mu_2)&=(\zeta_y, w_2)_{\mathcal T_h}\label{eq_zh_zhu}.
	\end{align}
\end{subequations}

\begin{lemma}
	If \textbf{(A1)} and \textbf{(A2)} hold, then
	\begin{align*}
	\gamma\norm{u-u_h}_{\varepsilon_h^\partial}^2  + \norm {\zeta_y}_{\mathcal T_h}^2 &= \langle \gamma u+\bm p_h(u)\cdot\bm n +h^{-1}  z_h(u) + \tau_2 z_h(u),u-u_h\rangle_{\varepsilon_h^\partial} \\
	&\quad- \langle \gamma u_h+\bm p_h\cdot\bm n + h^{-1}  z_h+\tau_2 z_h,u-u_h\rangle_{\varepsilon_h^\partial}.
	\end{align*}
\end{lemma}
\begin{proof}
	We have
	\begin{multline*}
	\langle \gamma u+\bm p_h(u)\cdot\bm n +h^{-1}  z_h(u) + \tau_2 z_h(u),u-u_h\rangle_{\varepsilon_h^\partial} - \langle \gamma u_h+\bm p_h\cdot\bm n + h^{-1}  z_h+\tau_2 z_h,u-u_h\rangle_{\varepsilon_h^\partial}\\
	= \gamma\norm{u-u_h}_{\varepsilon_h^\partial}^2 +\langle \zeta_{\bm p}\cdot\bm n+h^{-1}  \zeta_z +\tau_2 \zeta_z,u-u_h\rangle_{\varepsilon_h^\partial}.
	\end{multline*}
	Next, \Cref{identical_equa} gives
	\begin{align*}
	\mathscr B_1 &(\zeta_{\bm q},\zeta_y,\zeta_{\widehat{y}};\zeta_{\bm p},-\zeta_{z},-\zeta_{\widehat z}) + \mathscr B_2(\zeta_{\bm p},\zeta_z,\zeta_{\widehat z};-\zeta_{\bm q},\zeta_y,\zeta_{\widehat{y}})  = 0.
	\end{align*}
	Also, since $\tau_2$ is piecewise constant on $\partial \mathcal T_h$ we have
	\begin{align*}
	\mathscr B_1(\zeta_{\bm q},\zeta_y,\zeta_{\widehat y};&\zeta_{\bm p}, -\zeta_{z},-\zeta_{\widehat{ z}}) + \mathscr B_2(\zeta_{\bm p},\zeta_z,\zeta_{\widehat z}; -\zeta_{\bm q}, \zeta_{y},\zeta_{\widehat{y}})\\
	&=  (\zeta_{ y},\zeta_{ y})_{\mathcal{T}_h} - \langle P_Mu-u_h, \zeta_{\bm p}\cdot \bm{n} + (h^{-1} +\tau_1-\bm{\beta} \cdot\bm n)\zeta_z \rangle_{{\varepsilon_h^{\partial}}}\\
	&= (\zeta_{ y},\zeta_{ y})_{\mathcal{T}_h} -\langle P_Mu-u_h, \zeta_{\bm p}\cdot \bm{n} + h^{-1}\zeta_z +\tau_2 \zeta_z  \rangle_{{\varepsilon_h^{\partial}}}\\
	&= (\zeta_{ y},\zeta_{ y})_{\mathcal{T}_h} -\langle u-u_h, \zeta_{\bm p}\cdot \bm{n} + h^{-1}\zeta_z +\tau_2 \zeta_z  \rangle_{{\varepsilon_h^{\partial}}}.
	\end{align*}
	The above equalities yield
	\begin{align*}
	(\zeta_{ y},\zeta_{ y})_{\mathcal{T}_h} = \langle u-u_h, \zeta_{\bm p}\cdot \bm{n} + h^{-1}\zeta_z +\tau_2 \zeta_z  \rangle_{{\varepsilon_h^{\partial}}}.
	\end{align*}
\end{proof}

\begin{theorem}
	We have
	\begin{align*}
	\norm{u-u_h}_{\varepsilon_h^\partial}&\lesssim h^{s_{\bm p}-\frac 1 2}\norm{\bm p}_{s_{\bm p},\Omega} +  h^{s_{z}-\frac 3 2}\norm{z}_{s_{z},\Omega} + h^{s_{\bm q}+\frac 1 2}\norm{\bm q}_{s_{\bm q},\Omega} + h^{s_{y}-\frac 12 }\norm{y}_{s_{y},\Omega},\\
	\norm {y-y_h}_{\mathcal T_h} &\lesssim h^{s_{\bm p}-\frac 1 2}\norm{\bm p}_{s_{\bm p},\Omega} +  h^{s_{z}-\frac 3 2}\norm{z}_{s_{z},\Omega} + h^{s_{\bm q}+\frac 1 2}\norm{\bm q}_{s_{\bm q},\Omega} + h^{s_{y}-\frac 12 }\norm{y}_{s_{y},\Omega}.
	\end{align*}
\end{theorem}

\begin{proof}
	%	Since $u+\gamma^{-1}\bm p \cdot\bm n=0$, $u_h+\gamma^{-1}\bm p_h\cdot\bm n +\gamma^{-1}h^{-1}P_M z_h=0$,  and $\widehat z_h(u) = z=0$ all on $\varepsilon_h^{\partial}$, we have
	%
	The optimality conditions yield $\gamma u+\bm p \cdot\bm n=0$ and $\gamma u_h+\bm p_h\cdot\bm n + h^{-1} z_h +\tau_2 z_h=0$ on $\varepsilon_h^{\partial}$.  Therefore, the above lemma gives
	\begin{align*}
	\gamma\norm{u-u_h}_{\varepsilon_h^\partial}^2  + \norm {\zeta_y}_{\mathcal T_h}^2 &= \langle \gamma u+\bm p_h(u)\cdot\bm n + h^{-1} z_h(u) + \tau_2 z_h(u),u-u_h\rangle_{\varepsilon_h^\partial}\\
	&=\langle (\bm p_h(u)-\bm p)\cdot\bm n + h^{-1} z_h(u)+ \tau_2 z_h(u) ,u-u_h\rangle_{\varepsilon_h^\partial}.
	\end{align*}
	Since $\widehat z_h(u) = z=0$ on $\varepsilon_h^{\partial}$, we have
	\begin{align*}
	\norm {\bm p_h(u)-\bm p}_{\partial \mathcal T_h} &\le \norm {\bm p_h(u)-\bm{\Pi}\bm p}_{\partial \mathcal T_h} +\norm {\bm{\Pi}\bm p - \bm p}_{\partial \mathcal T_h}\\
	&\lesssim h^{-\frac 1 2}\norm {\varepsilon_h^{\bm p}}_{\mathcal T_h} +h^{s_{\bm p}-\frac 12 } \norm{\bm p}_{s^{\bm p},\Omega},\\
	\|z_h(u)\|_{\varepsilon_h^\partial} &=\|z_h(u) -\Pi z +\Pi z - z +P_M z-\widehat z_h(u) \|_{\varepsilon_h^\partial} \\
	&\le   \|\varepsilon_h^z -\varepsilon_h^{\widehat z}\|_{\partial\mathcal T_h} + \|\Pi z - z\|_{\varepsilon_h^\partial}.
	\end{align*}
	This implies
	\begin{align*}
	\norm{u-u_h}_{\varepsilon_h^\partial}  + \|\zeta_y\|_{\mathcal T_h}&\lesssim h^{-\frac 1 2}\norm {\varepsilon_h^{\bm p}}_{\mathcal T_h} +h^{s_{\bm p}-\frac 12 } \norm{\bm p}_{s^{\bm p},\Omega} \\
	&\quad+h^{-1}\|\varepsilon_h^z -\varepsilon_h^{\widehat z}\|_{\partial\mathcal T_h} + h^{-\frac 3 2} \|\delta^z\|_{\mathcal T_h}.
	\end{align*}
	\Cref{lemma:step5_main_lemma} and approximation properties of the $ L^2 $ projection give
	\begin{align*}
	\hspace{3em}&\hspace{-3em} \norm{u-u_h}_{\varepsilon_h^\partial}  + \|\zeta_y\|_{\mathcal T_h}\\
	&\lesssim h^{s_{\bm p}-\frac 1 2}\norm{\bm p}_{s^{\bm p},\Omega} + h^{s_{z}-\frac 3 2}\norm{z}_{s^{z},\Omega} +h^{s_{\bm q}+\frac 1 2}\norm{\bm q}_{s^{\bm q},\Omega} + h^{s_{y}-\frac 1 2}\norm{y}_{s^{y},\Omega}.
	\end{align*}
	The triangle inequality and \Cref{lemma:step3_conv_rates} yield
	\begin{align*}
	\|y -y_h\|_{\mathcal T_h}\lesssim h^{s_{\bm p}-\frac 1 2}\norm{\bm p}_{s^{\bm p},\Omega} + h^{s_{z}-\frac 3 2}\norm{z}_{s^{z},\Omega} +h^{s_{\bm q}+\frac 1 2}\norm{\bm q}_{s^{\bm q},\Omega} + h^{s_{y}-\frac 1 2}\norm{y}_{s^{y},\Omega}.
	\end{align*}
\end{proof}

\subsubsection{Step 7: Estimates for $\|\boldmath p-\boldmath p_h\|_{\mathcal T_h}$,  $\|z-z_h\|_{\mathcal T_h}$, and $\|\boldmath q - \boldmath q_h\|_{\mathcal T_h}$}
\begin{lemma}
	We have
	\begin{align*}
	\norm {\zeta_{\bm p}}_{\mathcal T_h}  &\lesssim h^{s_{\bm p}-\frac 1 2}\norm{\bm p}_{s_{\bm p},\Omega} +  h^{s_{z}-\frac 3 2}\norm{z}_{s_{z},\Omega} + h^{s_{\bm q}+\frac 1 2}\norm{\bm q}_{s_{\bm q},\Omega} + h^{s_{y}-\frac 12 }\norm{y}_{s_{y},\Omega},\\
	\|\zeta_z\|_{\mathcal T_h} & \lesssim  h^{s_{\bm p}-\frac 1 2}\norm{\bm p}_{s_{\bm p},\Omega} +  h^{s_{z}-\frac 3 2}\norm{z}_{s_{z},\Omega} + h^{s_{\bm q}+\frac 1 2}\norm{\bm q}_{s_{\bm q},\Omega} + h^{s_{y}-\frac 12 }\norm{y}_{s_{y},\Omega}.
	\end{align*}
\end{lemma}
\begin{proof}
	By the energy identity for $ \mathscr B_2 $ in \Cref{property_B}, the second error equation \eqref{eq_zh_zhu}, and since $\zeta_{ \widehat z} = 0$ on $\varepsilon_h^{\partial}$, we have
	\begin{align*}
	\hspace{2em}&\hspace{-2em}\mathscr B_2(\zeta_{\bm p},\zeta_z,\zeta_{\widehat z};\zeta_{\bm p},\zeta_z,\zeta_{\widehat z})\\
	&=(\zeta_{\bm p}, \zeta_{\bm p})_{{\mathcal{T}_h}}+\langle (h^{-1}+\tau_2+\frac 12 \bm{\beta}\cdot\bm n) (\zeta_z-\zeta_{\widehat z}) , \zeta_z-\zeta_{\widehat z}\rangle_{\partial{{\mathcal{T}_h}}}\\
	&=(\zeta_y,\zeta_z)_{\mathcal T_h}\\
	&\le \norm{\zeta_y}_{\mathcal T_h} \norm{\zeta_z}_{\mathcal T_h}\\
	&\lesssim \norm{\zeta_y}_{\mathcal T_h} (\|\nabla \zeta_z\|_{\mathcal T_h} + h^{-\frac 1 2} \|\zeta_z - \zeta_{\widehat z}\|_{\partial\mathcal T_h}) \\
	& \lesssim \norm{\zeta_y}_{\mathcal T_h}
	(\|\zeta_{\bm p}\|_{\mathcal T_h}+h^{-\frac1 2}\|\zeta_z-\zeta_{\widehat z}\|_{\partial\mathcal T_h}).
	\end{align*}
	Here, for the last two inequalities we used the discrete Poincar{\'e} inequality in \\ \Cref{lemma:discr_Poincare_ineq} and \Cref{nabla_ine}.  This gives
	% we get
	\begin{align*}
	\hspace{3em}&\hspace{-3em} \norm {\zeta_{\bm p}}_{\mathcal T_h} +h^{-\frac1 2}\|\zeta_z-\zeta_{\widehat z}\|_{\partial\mathcal T_h}\\
	& \lesssim h^{s_{\bm p}-\frac 1 2}\norm{\bm p}_{s_{\bm p},\Omega} +  h^{s_{z}-\frac 3 2}\norm{z}_{s_{z},\Omega} + h^{s_{\bm q}+\frac 1 2}\norm{\bm q}_{s_{\bm q},\Omega} + h^{s_{y}-\frac 12 }\norm{y}_{s_{y},\Omega}.
	\end{align*}
	Using the discrete Poincar{\'e} inequality and \Cref{nabla_ine} again yields
	\begin{align*}
	\|\zeta_z\|_{\mathcal T_h} & \lesssim \|\nabla \zeta_z\|_{\mathcal T_h} + h^{-\frac 1 2} \|\zeta_z - \zeta_{\widehat z}\|_{\partial\mathcal T_h}\\
	&\lesssim h^{s_{\bm p}-\frac 1 2}\norm{\bm p}_{s_{\bm p},\Omega} +  h^{s_{z}-\frac 3 2}\norm{z}_{s_{z},\Omega} + h^{s_{\bm q}+\frac 1 2}\norm{\bm q}_{s_{\bm q},\Omega} + h^{s_{y}-\frac 12 }\norm{y}_{s_{y},\Omega}.
	\end{align*}
\end{proof}

To obtain a positive convergence rate for $ \bm q $, we need% will assume the inequalities in \ref{eqn:s_rates_ineq} are strict, i.e.,
\begin{equation}\label{eqn:strict_regularity}
r_y > 1,  \quad  r_z > 2,  \quad  r_{\bm q} > 0,  \quad  r_{\bm p} > 1.
\end{equation}
\begin{lemma}
	If \textbf{(A1)}, \eqref{eqn:strict_regularity}, and $k\geq 1$ hold, then
	\begin{align*}
	\norm {\zeta_{\bm q}}_{\mathcal T_h}  &\lesssim h^{s_{\bm p}-1}\norm{\bm p}_{s_{\bm p},\Omega} +  h^{s_{z}-2}\norm{z}_{s_{z},\Omega} + h^{s_{\bm q}}\norm{\bm q}_{s_{\bm q},\Omega} + h^{s_{y}-1 }\norm{y}_{s_{y},\Omega}.
	\end{align*}
\end{lemma}
\begin{proof}
	By the energy identity in \Cref{property_B}, the first error equation \eqref{eq_yh_yhu}, and since $ \tau_2 $ is piecewise constant on $ \partial \mathcal{T}_h $, we have
	\begin{align*}
	\hspace{2em}&\hspace{-2em}\mathscr B_1(\zeta_{\bm q},\zeta_y,\zeta_{\widehat y};\zeta_{\bm q},\zeta_y,\zeta_{\widehat y}) \\
	&=(\zeta_{\bm q}, \zeta_{\bm q})_{{\mathcal{T}_h}}+\langle (h^{-1}+\tau_1 - \frac 12 \bm \beta \cdot\bm n) (\zeta_y-\zeta_{\widehat y}) , \zeta_y-\zeta_{\widehat y}\rangle_{\partial{{\mathcal{T}_h}}\backslash\varepsilon_h^\partial} - (\nabla\cdot\bm{\beta} \zeta_y,\zeta_y)_{\mathcal T_h}\\
	&\quad+ \langle (h^{-1}+\tau_1 - \frac 12 \bm \beta \cdot\bm n) \zeta_y, \zeta_y \rangle_{\varepsilon_h^\partial} \\
	&= -\langle P_M u-u_h, \zeta_{\bm q}\cdot \bm{n} + (\bm{\beta}\cdot\bm n-h^{-1}-\tau_1) \zeta_y \rangle_{{\varepsilon_h^{\partial}}}\\
	&= -\langle P_Mu-u_h, \zeta_{\bm q}\cdot \bm{n} - (h^{-1}+\tau_2) \zeta_y \rangle_{{\varepsilon_h^{\partial}}}\\
	&=-\langle u-u_h, \zeta_{\bm q}\cdot \bm{n} - (h^{-1}+\tau_2) \zeta_y \rangle_{{\varepsilon_h^{\partial}}}\\
	&	\lesssim \norm {u-u_h}_{\varepsilon_h^{\partial}} (\norm {\zeta_{\bm q}}_{\varepsilon_h^{\partial}} + h^{-1} \norm {\zeta_{y}}_{\varepsilon_h^{\partial}} )\\
	&	\lesssim h^{-\frac 1 2}\norm {u-u_h}_{\varepsilon_h^{\partial}} (\norm {\zeta_{\bm q}}_{\mathcal T_h} + h^{-\frac 1 2} \norm {\zeta_{y}}_{\varepsilon_h^{\partial}}).
	\end{align*}
	This gives
	\begin{align*}
	\norm {\zeta_{\bm q}}_{\mathcal T_h} &\lesssim h^{-\frac 1 2}\norm {u-u_h}_{\varepsilon_h^{\partial}} \\
	&\lesssim h^{s_{\bm p}-1}\norm{\bm p}_{s_{\bm p},\Omega} +  h^{s_{z}-2}\norm{z}_{s_{z},\Omega} + h^{s_{\bm q}}\norm{\bm q}_{s_{\bm q},\Omega} + h^{s_{y}-1 }\norm{y}_{s_{y},\Omega}.
	\end{align*}
\end{proof}

The above lemma, the triangle inequality, \Cref{lemma:step3_conv_rates}, and \Cref{lemma:step6_conv_rates} complete the proof of the main result:
\begin{theorem}
	We have
	\begin{align*}
	\norm {\bm p - \bm p_h}_{\mathcal T_h}   &\lesssim h^{s_{\bm p}-\frac 1 2}\norm{\bm p}_{s_{\bm p},\Omega} +  h^{s_{z}-\frac 3 2}\norm{z}_{s_{z},\Omega} + h^{s_{\bm q}+\frac 1 2}\norm{\bm q}_{s_{\bm q},\Omega} + h^{s_{y}-\frac 12 }\norm{y}_{s_{y},\Omega},\\
	\norm {z - z_h}_{\mathcal T_h}   & \lesssim  h^{s_{\bm p}-\frac 1 2}\norm{\bm p}_{s_{\bm p},\Omega} +  h^{s_{z}-\frac 3 2}\norm{z}_{s_{z},\Omega} + h^{s_{\bm q}+\frac 1 2}\norm{\bm q}_{s_{\bm q},\Omega} + h^{s_{y}-\frac 12 }\norm{y}_{s_{y},\Omega}.
	\end{align*}
	If in addition \eqref{eqn:strict_regularity} is satisfied and $ k \geq 1 $, then
	\begin{align*}
	\norm {\bm q - \bm q_h}_{\mathcal T_h} &\lesssim h^{s_{\bm p}-1}\norm{\bm p}_{s_{\bm p},\Omega} +  h^{s_{z}-2}\norm{z}_{s_{z},\Omega} + h^{s_{\bm q}}\norm{\bm q}_{s_{\bm q},\Omega} + h^{s_{y}-1 }\norm{y}_{s_{y},\Omega}.
	\end{align*}
\end{theorem}

\section{Numerical Experiments}
\label{sec:numerics}

We present numerical results for a 2D example problem similar to examples from \cite{MR2272157,MR2806572} with $ \bm \beta = \bm 0 $.  We consider a square domain $\Omega = [0,1/8]\times [0,1/8] \subset \mathbb{R}^2$, and choose the problem data
\begin{align*}
f=0, \ \  y_d = (x^2+y^2)^{-1/3},\ \  \bm \beta = [1, 1], \ \ \ \mbox{and} \  \ \ \gamma = 1.
\end{align*}
The largest interior angle is $ \omega = {\pi}/{2}$, and therefore $ r_\Omega = 3/2 $.  Also, we have $y_d\in H^{1/3-\varepsilon}(\Omega) $ for any $\varepsilon> 0$, and therefore $ r_d = 5/6-\varepsilon $ for any $\varepsilon> 0$.  For this example, the value of $ r_d $ restricts the guaranteed regularity of the solution.

We do not have an exact solution for this problem; therefore, we generate numerical convergence rates by computing errors between approximate solutions computed on different meshes.  Specifically, we compare approximate solutions computed on various meshes with the approximate solution on a fine mesh with 524288 elements, i.e., $h = 2^{-12}\sqrt 2$.  For all computations, we take $ \tau_2 = 1 $ and $ \tau_1 = \tau_2 + \bm \beta \cdot \bm n $ so that \textbf{(A1)}-\textbf{(A3)} are satisfied.

When $k=1$, the guaranteed theoretical convergence rates are given by \Cref{cor:main_result} in \Cref{sec:analysis}:
\begin{align*}
&\norm{y-{y}_h}_{0,\Omega}=O( h^{5/6-\varepsilon} ),  &  &\norm{z-{z}_h}_{0,\Omega} = O( h^{5/6-\varepsilon} ),\\
&\norm{\bm{q}-\bm{q}_h}_{0,\Omega} = O( h^{1/3-\varepsilon} ),  &  &\norm{\bm{p}-\bm{p}_h}_{0,\Omega} = O( h^{5/6-\varepsilon} ),
\end{align*}
and
\begin{align*}
&\norm{u-{u}_h}_{0,\Gamma} = O( h^{5/6-\varepsilon}).
\end{align*}
\Cref{table_1} shows numerical results for this case.  As in Part I, the numerically observed convergence rates match the theory for the control $ u $ and the primary flux $ \bm q $, but are higher than the theoretical rates for the other variables.  As mentioned in Part~I, similar convergence behavior has been observed in other works \cite{HuShenSinglerZhangZheng_HDG_Dirichlet_control1,MR3070527,MR3317816,MR2806572}.
\begin{table}%[!hbp]
	\begin{center}
		\begin{tabular}{|c|c|c|c|c|c|}
			\hline
			$h/\sqrt{2}$ &$2^{-4}$& $2^{-5}$&$2^{-6}$&$2^{-7}$ & $2^{-8}$ \\
			\hline
			$\norm{\bm{q}-\bm{q}_h}_{0,\Omega}$&1.45e-1   &1.00e-1  &7.41e-2   &5.63e-2& 4.30e-2 \\
			\hline
			order&-& 0.53& 0.44  &0.40&  0.39\\
			\hline
			$\norm{\bm{p}-\bm{p}_h}_{0,\Omega}$&2.67e-3  &9.65e-4   &3.55e-4   &1.35e-4 &5.20e-5\\
			\hline
			order&-&  1.47&1.44 &1.40 & 1.37 \\
			\hline
			$\norm{{y}-{y}_h}_{0,\Omega}$&1.00e-3   &3.32e-4   &1.21e-4   &4.60e-5 & 1.80e-5\\
			\hline
			order&-& 1.60&1.46&1.39 & 1.35 \\
			\hline
			$\norm{{z}-{z}_h}_{0,\Omega}$& 5.91e-5   &1.21e-5  &2.43e-6   &4.84e-7 & 9.63e-8 \\
			\hline
			order&-& 2.29&2.32&2.33& 2.33 \\
			\hline
			$\norm{{u}-{u}_h}_{0,\Gamma}$&1.31e-2& 6.38e-3&3.32e-3&1.81e-3 & 1.00e-3 \\
			\hline
			order&-&  1.03& 0.94&0.88& 0.85 \\
			\hline
		\end{tabular}
	\end{center}
	\caption{2D Example with $k=1$: Errors for the control $u$, state $y$, adjoint state $z$, and the fluxes $\bm q$ and $\bm p$.}\label{table_1}
\end{table}

Next, for $k=0$, \Cref{cor:main_result} gives the suboptimal convergence rates
\begin{align*}
\norm{y-{y}_h}_{0,\Omega}=O( h^{1/2-\varepsilon} ),~~  \;\norm{z-{z}_h}_{0,\Omega}=O( h^{1/2-\varepsilon} ),~~ \norm{\bm{p}-\bm{p}_h}_{0,\Omega} = O( h^{1/2-\varepsilon} ),
\end{align*}
and
\begin{align*}
&\norm{u-{u}_h}_{0,\Gamma} = O( h^{1/2-\varepsilon}).
\end{align*}
As in Part I, we observe much larger numerical convergence rates for all variables.  Improving the analysis for the $ k = 0 $ case is again an interesting topic we leave to be considered elsewhere.
\begin{table}%[!hbp]
	\begin{center}
		\begin{tabular}{|c|c|c|c|c|c|}
			\hline
			$h/\sqrt{2}$ &$2^{-4}$& $2^{-5}$&$2^{-6}$&$2^{-7}$ & $2^{-8}$ \\
			\hline
			$\norm{\bm{q}-\bm{q}_h}_{0,\Omega}$&2.22e-1   &1.69e-1   &1.22e-1   &8.92e-2   &6.56e-2 \\
			\hline
			order&-&  0.39   &0.47   &0.46   &0.44\\
			\hline
			$\norm{\bm{p}-\bm{p}_h}_{0,\Omega}$&8.60e-3   &5.10e-3   &2.75e-3   &1.43e-3   &7.31e-4\\
			\hline
			order&-&  0.75   &0.90   &0.94   &0.97 \\
			\hline
			$\norm{{y}-{y}_h}_{0,\Omega}$&2.96e-3   &1.33e-3   &4.91e-4   &1.82e-4   &6.97e-5 \\
			\hline
			order&-& 1.15   &1.44   &1.43   &1.39 \\
			\hline
			$\norm{{z}-{z}_h}_{0,\Omega}$&    3.82e-4   &1.08e-4   &2.89e-5   &7.48e-6   &1.90e-6 \\
			\hline
			order&-& 1.82   &1.91   &1.95   &1.97 \\
			\hline
			$\norm{{u}-{u}_h}_{0,\Gamma}$&2.83e-2   &1.79e-2   &1.07e-2   &6.14e-3   &3.47e-3 \\
			\hline
			order&-&  0.66  &0.75   &0.80   &0.82\\
			\hline
		\end{tabular}
	\end{center}
	\caption{2D Example with $k=0$: Errors for the control $u$, state $y$, adjoint state $z$, and the fluxes $\bm q$ and $\bm p$.}\label{table_2}
\end{table}

\section{Conclusion}

In Part I of this work, we considered a Dirichlet boundary control problem for an elliptic convection diffusion equation and approximated the solution using a new HDG method.  We also proved optimal convergence rates for the control under a high regularity assumption.  In this paper, we removed the restrictions on the domain $\Omega$ and the desired state $y_d$ from Part I and considered a low regularity scenario.  We used very different HDG analysis techniques to prove optimal convergence rates for the control.

As far as we are aware, this paper and Part I are the only existing analysis and numerical analysis explorations of this convection diffusion Dirichlet control problem.  We leave many topics to be considered in future work, such as improving the HDG convergence analysis for the Dirichlet boundary control problem considered here and also applying HDG methods to Dirichlet control problems for fluids.

\bibliographystyle{siamplain}
\bibliography{yangwen_ref_papers,yangwen_ref_books}

\end{document}